\documentclass[12pt]{article}
\usepackage{amsfonts}
\usepackage{amsmath,latexsym,amssymb,amsthm,enumerate,amsmath,amscd}
\usepackage[mathscr]{eucal}
\usepackage{authblk}
\usepackage{bbm,color}
\usepackage{tikz-cd,rotating}
\usepackage{pdflscape}
\usepackage[utf8]{inputenc}
\usepackage[all,cmtip]{xy}
\usepackage{graphicx}
\usepackage{ytableau}
\usepackage {tikz}
\usepackage {cite}
\usepackage{comment}
\usepackage{array}
\usepackage{hyperref}
\usepackage{a4wide,floatflt,multicol,graphicx}
\theoremstyle{plain}
\newtheorem{theorem}{Theorem}[section]
\newtheorem{proposition}[theorem]{Proposition}
\newtheorem{corollary}[theorem]{Corollary}
\newtheorem{lemma}[theorem]{Lemma}
\theoremstyle{definition}
\newtheorem{definition}[theorem]{Definition}

\newtheorem*{ack}{Acknowledgment}

\newtheorem{remark}[theorem]{Remark}

\newtheorem{example}[theorem]{Example}

\newtheorem{question}[theorem]{Question}
\newtheorem{thm}{Theorem}
\def\cha{\mathrm{char}\ }

\def\Hilb{\mathrm{Hilb}}

\def\rk{\mathrm{rk}}
\def\Gor{\mathrm{Gor}}

\def\<{\left<}
\def\>{\right>}

\def\Max{\mathrm{Max}}

\def\Z{\mathfrak{Z}}

\def\Ann{\mathrm{Ann}}
\def\Sym{\mathrm{Sym}}

\title{Jordan degree type for codimension three Gorenstein algebras of almost constant Hilbert function} 
\begin{document}
\author[1]{Nancy Abdallah}
\author[2]{Nasrin Altafi}
\author[3]{Anthony Iarrobino}
\author[4]{Joachim Yam\'{e}ogo}
\affil[1]{\small University of Gothenburg, Gothenburg, Sweden.} 
\affil[2]{Department of Mathematics and Statistics, Queen's University, Kingston, Ontario, Canada, and Department of Mathematics, KTH Royal Institute of Technology, Stockholm, Sweden.}
\affil[3]{Department of Mathematics, Northeastern University, Boston, MA 02115, USA.}
\affil[4]{Universit\'e C\^ote d'Azur, CNRS, LJAD, France.}

\title{Jordan degree type for codimension three Gorenstein algebras of small Sperner number} 

\renewcommand\footnotemark{}
\thanks{ \textbf{Email~addresses:} {nancya@chalmers.se, nasrin.altafi@gmail.com, a.iarrobino@northeastern.edu, joachim.yameogo@univ-cotedazur.fr}}

\date{August 22, 2025}
\maketitle

\begin{abstract}  The Jordan type $P_{A,\ell}$ of a linear form $\ell$ acting on a graded Artinian algebra $A$ over a field $\sf k$ is the partition describing the Jordan block decomposition of the multiplication map $m_\ell$, which is nilpotent.
The Jordan degree type $\mathcal S_{A,\ell}$ is a finer invariant, describing also the initial degrees of the simple submodules of $A$ in a decomposition of $A$ as a direct sum of ${\sf k}[\ell]$-modules. The set of Jordan types of $A$ or Jordan degree types (JDT) of $A$ as $\ell$ varies,  is an invariant of the algebra. This invariant has been studied for codimension two graded algebras. We here extend the previous results to certain codimension three graded Artinian Gorenstein (AG) algebras - those of small Sperner number. Given a Gorenstein sequence $T$ - one possible for the Hilbert function of a codimension three graded AG algebra - the irreducible variety $\Gor(T)$ parametrizes all Gorenstein algebras of Hilbert function $T$. We here completely determine the JDT possible for all pairs $(A,\ell), A\in \Gor(T)$, for Gorenstein sequences $T$ of the form $T=(1,3,s^k,3,1)$ for Sperner number $s=3,4,5$ and arbitrary multiplicity $k$. For $s=6$ we delimit the prospective JDT, without verifying that each occurs. \par
\end{abstract}
\tableofcontents

\section{Introduction.}

Let $R={\sf k}[x,y,z]$, the polynomial ring over an infinite field $\sf k$.  We consider here codimension three standard graded Artinian Gorenstein (AG) algebra quotients $A=R/I$, where $ I$ is an ideal of $R$, that have {\it almost constant Hilbert functions} $H(A)=T$ with Sperner number - maximum value of the Hilbert function - at most 6. That is, $T$ satisfies
\begin{equation}\label{Taceq}
T=(1,3,s^k,3,1) \text{ with } 3\le s \le 6, \text{ and } 1\le k,
\end{equation}
where $j=k+3$ is the socle degree. For a graded Artinian algebra $A$, the \emph{Jordan type} $P_{A,\ell}$ where $\ell$ is a linear form in $R$ is the partition giving the Jordan block decomposition of the multiplication $m_\ell$ by $\ell$ on $A$, which is nilpotent.
The algebra $A$ has a decomposition as $\sf k[\ell]$ module: that is, $A$ is a direct sum of $t$ simple modules, each isomorphic to ${\sf k}[\ell]/(\ell^{p_i})$, where $P_{A,\ell}=(p_1,p_2,\ldots, p_t), p_1\ge p_2\ge \cdots \ge p_t$. The Jordan degree type (JDT) specifies also the degrees of the generators of these simple modules, which is an invariant of the pair $(A,\ell)$ (Definition \ref{JDT1def}).

Several authors have determined the possible Jordan types for codimension two graded complete intersections \cite{AIK}, codimension two Artinian algebras \cite{AIKY}, and certain graded Artinian Gorenstein (AG) algebras of low socle degree \cite{CGo}.  These studies determined the possible Jordan types for the pair $(A,\ell)$, for an arbitrary Artinian algebra $A$ in the class studied, and an arbitrary linear form $\ell$. For codimension three AG algebras, the possible Jordan types and Jordan degree types of these algebras are not well understood. \footnote{In higher codimension than three, the set of Gorenstein sequences themselves are not fully known, for example, in codimension four, whether they are unimodal, or satisfy an SI condition \cite{SeSr}.}

We here extend this study to height three graded AG algebras having Sperner number $ 3,4,5$ or $6$, and whose Hilbert function is almost constant. These JDT  correspond 1-1 to the prospective JDT matrices arising from the rank matrices studied by the second author in \cite{Al1} (see Section \ref{ranksec}). We introduce a new combinatorial method to determine the potential JDT matrices and their corresponding JDT for pairs $(A,\ell)$. For $s=3,4,5$, we prove that all these potential JDT are realizable. For $s=6$ we delimit the potential JDT, without verifying that each occurs: we conjecture that each such JDT is realizable. Our main result is the following: 

\begin{thm}\label{1thm} Let $T=(1,3,s^k,3,1)$ be almost constant (Equation \eqref{Taceq}), and assume that the field $\sf k$ has characteristic zero. 
 There are for $s=3$ and each $k\ge 3$ eight occurring Jordan degree types; for $s=4$ and each $k\ge 4$ there are twenty-six JDT; for $s=5$ and each $k\ge 3$ forty-seven JDT, and for $s=6$ and each $k$ at most sixty-five JDT.   When the field $\sf k$ has finite characteristic, there are at most the above number of rank matrices, or, equivalently, JDT.
\end{thm}

We show the above upper bounds on the number of rank matrices satisfying the conditions of Corollary \ref{Mcor} - thus the bounds on the numbers of JDT - in Theorem~\ref{s=3thm} and Table \ref{s=3tab} for $s=3$;  Theorem~\ref{4kthm} and Table \ref{s4tab} for $s=4$; Theorems~\ref{5kthm}, \ref{5kbthm} and Table \ref{s5tab} for $s=5$; and in Theorems \ref{s=6thm}, \ref{6bthm} and Table \ref{s6tab}  for $s=6$. When the characteristic $\cha \sf k = 0$ we specify an example for each JDT, with added data - usually a Macaulay dual generator and ideal - in the Tables \ref{1351fig}-\ref{newk=5btable} of Appendix \ref{tablesec}:  see Theorem~\ref{k3prop} for $s=3$,
Theorem \ref{s=4thm} for $s=4$, and Theorem \ref{s=5thm} for $s=5$.  In these results we also specify the occurring JDT for smaller multiplicities of $s$ than those in Theorem \ref{1thm}. For the case of finite characteristic see Remark~\ref{inffieldrem}. A notable feature of these results is that we give examples  for arbitrarily large $k$ or socle degree $j$.  In Section~\ref{exsec} we illustrate how we verified the tables, showing the several methods used. In Section~\ref{questsec} we pose further questions.  In particular, we conjecture that all potential rank matrices for codimension three Gorenstein sequences actually are realizable (Question  \ref{4.13quest}). \par A key step in the proof is that we show a pattern in the rank matrices for these Gorenstein sequences $T$ as in Equation \eqref{Gorseqeqn} when $k\ge 3$ (Lemma~\ref{diaglem}); the pattern arises from a projection morphism taking $\Gor(T)$ to the punctual Hilbert scheme $\Hilb^s(\mathbb P^2)$ (Lemma~\ref{maphilblem}). Applying this, in Section \ref{threesec} we first describe three kinds of growth of JDT sequences for a given Sperner number $s$, as the socle degree $j$ (or, equivalently, the multiplicity $k$ of $s$ in $T$) increases:  in particular we describe \emph{lengthening parts}, and \emph{repeating parts} (Definition~\ref{partsdef}). We then
show the connection between the $\Delta$ invariant of the rank matrix and the number of these parts (Proposition~\ref{partsprop}).  

\section{Rank matrices and Jordan degree type.}\label{rankjdtsec}
We describe codimension three Gorenstein sequences in Section \ref{Gorsec}, we introduce Lefschetz conditions in Section \ref{prelimsec}; we introduce Jordan basis and the Jordan degree type (JDT) in Section \ref{Jbsec} and in Section \ref{ranksec} establish the connection with rank matrices and the JDT matrix studied by the second author in \cite{Al1,Al3}. The map from $\Gor(T)$ to the punctual Hilbert scheme $\Hilb^s(\mathbb P^2) $ (Lemma \ref{maphilblem}) imposes conditions on a new $\Delta$ sequence associated to rank matrices for the more general Hilbert function of Equation \eqref{Gorseqeqn} (Corollary~\ref{deltacor}), and in particular for the almost constant Hilbert functions of Equation \eqref{Taceq} (Corollary~\ref{Mcor}). After introducing the Macaulay duality (Section~\ref{Macsec}) we specify the construction of the JDT matrix (Section \ref{JDTmatrixsec}) that we use throughout. Remark \ref{JDTmatrixtoJDTrem} recalls how to write down the Jordan degree type from the JDT matrix \cite{Al3}.  A theme of the paper is to restrict the potential JDT matrices, given the Gorenstein sequence $T$.

\subsection{Gorenstein sequences.}\label{Gorsec}
 A Hilbert function sequence $T=(1,t_1,\ldots, t_j)$ that is possible for a graded AG algebra is termed a \emph{Gorenstein sequence}.  Those occurring for a codimension three graded Artinian Gorenstein algebra satisfy
\begin{equation}\label{Gorseqeqn} T=(1,3,T_1,s^k,T_1^r,3,1),
\end{equation}
  where $T_1^r$ is the reverse of $T_1$ and $(1,3,T_1,s)$ is an O-sequence,
meaning that its first difference occurs as the Hilbert function of a codimension two graded Artinian algebra \cite{BuEi, St, Di}.\footnote{See \cite[\S 4.2, Theorem 4.2.10]{BruH} for a discussion of Macaulay growth conditions for Hilbert functions.} The maximum value $s$ of $T$ is termed the \emph{Sperner number} of $T$.  The sequences $T$ satisfying Equation \eqref{Gorseqeqn} are also termed codimension three SI sequences. The graded codimension three AG algebra quotients of $R$ having Hilbert function $T$ are parametrized by a smooth irreducible variety $\Gor(T)$ of known dimension  (see \cite{Di},\cite[Theorem 5.25]{IK}).

\subsection{Preliminaries, weak and strong Lefschetz.}\label{prelimsec}

We say that the pair $(A,\ell)$ where $A$ is a graded Artinian algebra and $\ell\in A$ is a linear form is a \emph{weak Lefschetz pair} (WL) if the multiplication map by $\ell$ from $A_i$ to $A_{i+1}$ has maximal rank for any non-negative integer $i$.  If for any pair of non-negative integers $(u,i)$ the multiplication map by $\ell^u$ from $A_i$ to $A_{i+u}$ has maximal rank, we say that the pair $(A,\ell)$ is \emph{strong Lefschetz} (SL). For an Artinian algebra $A$, if there exists a linear form $\ell$ such that $(A,\ell)$ is WL (or SL) pair we say $A$ satisfies the WL (or SL) property and $\ell$ is a WL (or SL) element. The weak and strong Lefschetz pairs $(A,\ell)$ can be determined from the Jordan type $P_{A,\ell}$. A pair $(A,\ell)$ is WL if and only if the number of parts of $P_{A,\ell}$ is the Sperner number (maximum value) of the Hilbert function $T$. A pair is SL if and only if $P_{A,\ell}$ is the conjugate $T^\vee $ of $T$, regarded as a partition (see \cite[Prop. 2.10]{IMM}). Artinian algebras of codimension two satisfy the SL property when the characteristic of $\sf k$ is zero or greater than the socle degree \cite{Br}. In particular, Artinian Gorenstein algebras of codimension two (these are complete intersections) over a field of characteristic zero have the SL property. We know that all codimension three graded {\it complete intersections} are weak Lefschetz when $\cha {\sf k}=0$, a result that already uses the Grauert M\"{u}lich theorem for vector bundles on $\mathbb P^2$ \cite{HMNW}. The question of whether all codimension three AG algebras $A$ have the weak Lefschetz property - equivalently, whether for a generic linear form $\ell$ the partition $P_{A,\ell}$ has the Sperner number of parts - has been reduced to the case of $A$  compressed of odd socle degree \cite{BMMNZ}.\footnote{Recall that a compressed AG algebra quotient $A$ of ${\sf k}[x,y,z]$ is one having the maximum possible Hilbert function given the soclel degree $j$.}  In a recent paper, the authors with A.~Seceleanu showed that when $\cha {\sf k}=0$, all codimension three AG algebras of Sperner number at most six are strong Lefschetz \cite{AA-Y}.  The second author has shown that for any SI Gorenstein sequence $T$ as in Equation \eqref{Gorseqeqn} - or its higher codimensional analogue - there is a strong Lefschetz AG algebra $A\in \Gor(T)$ \cite{Al2}; T. Harima had shown an analogous result for weak Lefschetz AG algebras $A$ in 1995 \cite{Ha}.  The codimension of the non weak-Lefschetz locus is studied in \cite{BMMN1}, and further results about weak-Lefschetz properties for AG algebras with small Sperner numbers are shown in \cite{BMMN2}. Here, in contrast, our goal is to specify all JDT consistent with each Hilbert function we study, analogous to the cited results for codimension two \cite{AIK,AIKY} and for Perazzo algebras \cite{MMP}.

\subsection{Jordan basis and Jordan degree type.}\label{Jbsec}
For the following definition of Jordan basis, JDT, see \cite[Definition 2.28(i)]{IMM}, \cite[\S 2.3]{AIM}, or \cite{Al3}. That the JDT can be defined for any graded Artinian algebra is a consequence of results of T.~Harima and J. Watanabe \cite{HW} concerning the invariants of central simple modules - see \cite[Lemma 2.2iv, Def. 2.28]{IMM}, also \cite[Definition~4.1]{CGo}. A further discussion of this connection is in \cite[Definition 2.33 and Lemma~2.34]{IMM}.  For a graded Gorenstein algebra $A$, the JDT with respect to a linear form $\ell\in R_1$ is symmetric (Lemma \ref{JDTsymlem}), making the JDT a more natural invariant than just the Jordan type of $A$.
\begin{definition}[Jordan basis, JDT]\label{JDT1def}
A. Let $A=R/I$ be an Artinian graded algebra, let $M$ be a finite graded $A$-module, and let $\ell\in A_1$ be a homogeneous linear element of $A$.  Then $M$ may be written as a direct sum of simple ${\sf k} [\ell]$ modules $M(i)$ that we call $\ell$-strings,   
\begin{equation}\label{decompeq} M=\oplus M(i) \text { where }M(i)=\langle v_i,\ell v_i, \ldots ,\ell^{p_i-1}v_i\rangle, 1\le i\le t, \text { satisfying } \ell^{p_i}v_i=0.
\end{equation}
We call this set of vector spaces ${\sf k}[\ell]v_i$ a \emph{Jordan basis} of $M$, with respect to $\ell$. The set of lengths of these strings is uniquely determined by $(M,\ell)$ and form the partition
$P_{M,\ell}$ with $t$ parts,  the \emph{Jordan type} of $(M,\ell)$. These are simply the dimensions of the Jordan blocks of the multiplication $m_\ell$ on $M$.

B.  The \emph{Jordan degree type} $\mathcal S_{M,\ell}$ supplements
the Jordan type by specifying the degrees  $\nu_i$ of the generators $v_i$ of each ${\sf k}[\ell]$-module $M(i)$. Given the decomposition of $M$ in Equation \eqref{decompeq} we set
$$\mathcal S_{M,\ell}=\{(p_i,\nu_i), \nu_i=\text { degree of } v_i.\} $$
We may also write the Jordan degree type in list manner:
\begin{equation}\label{listJDTeq}
\mathcal S_{M,\ell}=\sum_{p}\sum_\nu (p,\nu)^{\eta(p,\nu)}, 
\end{equation}
where $\eta(p,\nu)$ is the multiplicity of $(p,\nu)$, and the sum is over distinct pairs.\par
For convenience we may write the pair $(p,\nu)$ as $p_{\nu}$. Also we will write $p\uparrow_\nu^{\nu+\alpha}$ in place of $\sum_{i=0}^{\alpha}(p,\nu+i)$. 
\end{definition}
Although the decomposition of the ${\sf k}[\ell]$-module $ M$ into simples is not unique, for $M$ graded the JDT is uniquely determined up to order by the pair $(M,\ell)$ \cite[Lemma 2.2(iv)]{IMM}.\par
There is a notable symmetry for Jordan degree types that occur for standard graded Gorenstein algebras. This was first shown by T. Harima and J.~Watanabe \cite{HW}; it is described also in \cite[Lemma~4.6]{CGo}, and, in the language of JDT, in \cite[Proposition~2.28]{IMM}. We take our statement from \cite[Proposition 2.21]{AIM}.
\begin{lemma}\label{JDTsymlem}(Symmetry of JDT) Let $A$ be a graded AG algebra and $\ell\in A_1$, and write the Jordan degree type in list manner 
$\mathcal S_{A,\ell}=\sum_{p}\sum_\nu (p,\nu)^{\eta(p,\nu)}$ as above. Then we have for $\nu\le j/2$
\begin{equation}\label{JDTsymeq}
\eta(p,\nu)=\eta(p,j+1-\nu-p).
\end{equation}
\end{lemma}
B. Costa and R. Gondim give a nice visualization of this symmetry in their string diagrams, see \cite[p. 11-12 and \S 5]{CGo}. 
\par
  The second author has specified a smallest socle-degree pair of graded AG algebras in codimension three having the same Jordan type, but different JDT: two AG algebras $\{A,B\}$ of  Hilbert function $T=(1,3,6,9,9,9,6,3,1)$ of socle degree $8$ \cite[Example 4.2]{Al3}. 
\begin{example}[Which symmetric $\mathcal S$ occur in codimension three?]\label{2.3ex}
Consider $\mathcal{S}_1=(\mathsf{5}_0,\mathsf{3}_1,\mathsf{3}_1,\mathsf{1}_2)$ and $\mathcal {S}_2=(\mathsf{5}_0,\mathsf{3}_1,\mathsf{2}_1,\mathsf {2}_2)$  that are two potential JDT having the needed symmetry for an AG algebra $A$ of Hilbert function $T=(1,3,4,3,1)$; and $\mathcal S_3=({\sf 4}_0,({\sf 2}_1)^2)$ is consistent with $T=(1,3,3,1)$. 
These three actually occur (Theorem \ref{s=4thm}  and table \ref{1351fig}, respectively). However, the sequence $\mathcal S_4=({\sf 4}\uparrow_0^2, {\sf 1}_{1,4})$ which is symmetric in the sense of Equation~\eqref{JDTsymeq} and consistent with  the Hilbert function $T=(1,3^4,1)$, does not actually occur as the JDT of a graded AG algebra, as we shall see (Theorem \ref{k3prop} and Table~\ref{1351fig}). See Figure \ref{Sfig}.
\end{example}

Given a potential JDT sequence $\mathcal S$, we denote by $T(In (\mathcal S))$ the Hilbert function of the initial elements of the strings of $\mathcal S$: that is $T(In (\mathcal S))_i$ is the number of initial elements of the strings of $\mathcal S$ having degree $i$. For $\mathcal{S}_1$ we have $T(In\,\mathcal S_1)=T(In\,\mathcal S_2)=(1,2,1,0)$ but $T(In\,\mathcal S_4)=(1,2,1,0,1)$.
\begin{lemma}\label{Soccurlem} A JDT sequence $\mathcal S$ can occur for a codimension three AG algebra only if $T(In\,\mathcal S)$ is an $O$-sequence in codimension two. That is, $T(In\,\mathcal S)$ must occur as the Hilbert function of a graded quotient of ${\sf k}[y,z]$.
\end{lemma}
\begin{proof} Let $I\subset R={\sf k}[x,y,z]$ be an ideal such that the codimension three AG algebra $A=R/I$ has JDT $\mathcal S_{A,x}$.
One can compute a  Gr\"{o}bner basis (standard basis) of $I$ using homogeneous degree rvlex order with $z > y> x$. As stated in Definition \ref{JDT1def}, the algebra $A=R/I$ can be written as a direct sum: $A=\oplus A_m, A_m=\langle v_m,x v_m, \ldots ,x^{p_{\nu_m}-1}v_m\rangle, 1\le m\le t, \text { satisfying } x^{p_{\nu_m}}v_m=0$ and $\nu_m$ is the degree of the form $v_m$. Using homogeneous degree rvlex (reverse lex) order with $z > y> x$, let $w_m$ be the leading monomial of $v_m$. It is clear that $w_m\in {\sf k}[y,z]$. 
Since by definition, the set $\left\{x^{\alpha_m}v_m \right\}_{1\leq m\leq t, 0\leq\alpha_m\leq p_{\nu_m}-1}$ is a cobasis of the ideal $I$, we have that $\left\{w_m\right\}_{1\leq m \leq t}$ is a cobasis of the ideal $\left(I+(x)\right)$. Now, let $\gamma_i=\#\left\{w_m, degree(w_m)=i\right\}$. Then the sequence of the $\gamma_i$'s is the Hilbert function of the ideal $\left(I+(x)\right)$.
\end{proof}
The condition of Lemma \ref{Soccurlem} is not sufficient for a potential JDT sequence: that is, the rank conditions of Lemma \ref{ranklem} below impose further restrictions. 
\begin{example} Consider ${\mathcal S}=\left(5\uparrow_0^2,2_1,1_3,2_4\right)$. The sequence ${\mathcal S}$ is symmetric and compatible with the Hilbert function $(1,3,4^3,3,1)$. Let $\gamma_i$ be the number of strings of ${\mathcal S}$ starting at degree $i$. We have $\gamma_0=1, \gamma_1=2, \gamma_2=1, \gamma_3=1, \gamma_4=1$. That gives an $O$-sequence in codimension two: $(1,2,1,1,1)$. One can check in Table \ref{4ktable} that ${\mathcal S}=\left(5\uparrow_0^2,2_1,1_3,2_4\right)$ does not occur as JDT of an AG algebra, despite satisfying the condition of Lemma \ref{Soccurlem}.
\end{example}

\begin{question} Let $\mathcal S=\sum_{p}\sum_\nu (p,\nu)^{\eta(p,\nu)}$ be a symmetric sequence in the notation of Lemma \ref{JDTsymlem}. Let $p_1=\Max\left\{p\right\}$. Is it true that ${\mathcal S}$ can occur as a JDT if and only if for all $i\in [0, p_1] $, $\mathcal S=\sum_{p-i}\sum_\nu (p-i,\nu)^{\eta(p-i,\nu)}$ satisfies the property described in Lemma \ref{Soccurlem}: that is, if we shorten each string by $i$,  with the convention that we omit $(p-i,\nu)^{\eta(p-i,\nu)}$ whenever $p-i\leq 0$ then the initial terms of the strings determine an O-sequence?
\end{question} 
\par
Note that the question is based on the fact that if an ideal $I$ defines an algebra whose JDT with respect to $\ell=x$ is $\mathcal S=\sum_{p}\sum_\nu (p,\nu)^{\eta(p,\nu)}$, then the ideal quotient $(I:x)$ defines an algebra whose JDT will be $\mathcal S=\sum_{p-1}\sum_\nu (p-1,\nu)^{\eta(p-1,\nu)}$ (the length of each string is shortened by 1).\par

\begin{figure}
\begin{center}
\vskip 0.2cm
$\mathcal S_1:\begin{array}{ccccc}
*&*&*&*&*\\
&*&*&*&\\
&*&*&*&\\
&&*&&\\\hline
1&3&4&3&1\\
\end{array}$\qquad
$\mathcal S_2:\begin{array}{ccccc}
*&*&*&*&*\\
&*&*&*&\\
&*&*&&\\
&&*&*&\\\hline
1&3&4&3&1\\
\end{array}$\\
\vspace{0.5cm}
$\mathcal S_3:\begin{array}{ccccc}
	\\
	\\
*&*&*&*\\
&*&*&\\
&*&*&&\\\hline
1&3&3&1\\
&&&\\
\end{array}$\qquad 
$\mathcal S_4:\begin{array}{cccccc}
*&*&*&*&&\\
&*&*&*&*&\\
&&*&*&*&*\\
&*&&&*&\\\hline
1&3&3&3&3&1\\
\end{array}$

\normalsize\vskip 0.2cm
\caption{Sequences $\mathcal S_1,\mathcal S_2,\mathcal S_3$ occur as JDT (Table \ref{4ktable} \#1,3 and Table \ref{1351fig} \#1) but $\mathcal S_4$ does not (see Example \ref{2.3ex}).}\label{Sfig}
\end{center}
\end{figure}

\subsection{Rank matrices.}\label{ranksec}
A result of the second author (Lemma \ref{ranklem}, below, see \cite{Al1}) is key in determining the Jordan degree type of Artinian Gorenstein algebras. Assume that $T=(1,3,T_1,s^k,T_1^r,3,1)$ where $(1,3,T_1,s)$ has first difference an O-sequence, and $T_1^r$ is the reverse of $T_1$, as in Equation \eqref{Gorseqeqn}: any height three Gorenstein sequence can be written in this way.  For any AG algebra $A$ in $\Gor(T)$ we recall from \cite[Definition 3.1]{Al1} the rank matrix $M_{A,\ell}$ for the pair $(A,\ell)$ where $\ell\in A_1$. Here $M_{A,\ell}$ is the symmetric $(j+1)\times (j+1)$ upper triangular matrix whose entry $(M_{A,\ell})_{u,v}$ is the rank of $m_{\ell^{v-u}}: A_u \to A_v$ for $u\le v$.  We set for  $a\le j-b$,
 \begin{equation}\label{rankeq} 
 r_{a,b}(\ell)=\rk(m_{\ell^{j-b-a}}: A_a\to A_{j-b})=(M_{A,\ell})_{a,j-b}.
 \end{equation}

The following Lemma applies to all Gorenstein sequences - Hilbert functions of AG algebras - but we state it for codimension three.  We follow a notation in \cite{Al1}: given a length $w$ sequence $\sf v$ of integers, we denote by ${\sf v}_+$ the length $w+1$ sequence $(0,{\sf v})$. 
 \begin{lemma}[Rank Matrix conditions]\label{ranklem}\cite[Corollary 3.9]{Al1}
Let $T$ be a codimension three Gorenstein sequence as in Equation \eqref{Gorseqeqn}. The rank matrix $M_{A,\ell}$ 
satisfies the following: 
\begin{enumerate}[(i.)] 
\item
The diagonal parallel to the main diagonal through $(0,i)$ is the Hilbert function of the AG algebra $A(i)=R/(I:\ell^i)$ - so is a Gorenstein sequence.
\item The differences $(H(A(i))-H(A(i+1)_+), i=0,1,\ldots$ between adjacent diagonals of $M_{A,\ell}$ are O-sequences \cite[Lemma~3.6]{Al1}. \item The rows are non-increasing; the columns are non-decreasing within the upper triangular portion of $M_{A,\ell}$.
\item Furthermore, each adjacent square of entries in  $M_{A,\ell}$ on and above the main diagonal satisfies \cite[Lemma~3.7]{Al1}
\begin{equation}\label{sumdiageq}
\begin{pmatrix}\alpha&\beta\\\gamma&\zeta\end{pmatrix}\subset M_{A,\ell}\Rightarrow \gamma+\beta\ge \alpha+\zeta.
\end{equation} 
\end{enumerate}
\end{lemma}
\begin{example} We illustrate (ii). For $H(R/\Ann F)=(1,3,5,5,3,1)$, we cannot have $H(R/\Ann{(\ell\circ F)}=(1,1,1,1,1)$, as then $H(F)-H(\ell\circ F)_+=(1,3,5,5,3,1)-(0,1,1,1,1,1)=(1,2,4,4,2)$ which is not an O-sequence as $2_1$ cannot grow to $4_2$ \cite[Theorem 4.2.10]{BruH}.
\end{example}
We rely on \cite{Al1} for the proofs of Lemma \ref{ranklem} (i)-(iii). We will reprove Lemma \ref{ranklem} (iv) after introducing the Macaulay duality.

 \subsubsection{Macaulay duality.}\label{Macsec}

 \begin{definition}[Macaulay dual]\label{Macdef} We let the ring $R={\sf k}[x,y,z]$ act on $S={\sf k}[X,Y,Z]$ as
 contraction, that is
 \begin{equation}\label{contracteq}x^i\circ X^j=\begin{cases}X^{j-i}\text{ if $j\ge i$};\\
 $0$ \text{ otherwise,} \end{cases}
 \end{equation}
 similarly for $y^i\circ Y^j, z^i\circ Z^j$, to monomials of $R$ acting on monomials of $S$, and we extend linearly to an action of $R$ on $S$. Given $F\in S$ we denote by $I_F=\Ann(F)=\{h\in R\mid h\circ F=0\}$.  Given a vector subspace $V\subset R_j$ we denote by $V^\perp\subset S_j$ the space $V^\perp=\{w\in S_j\mid v\circ w=0$ for all $v\in V$\}.
 \end{definition}
  For example
 $x^2yz^2\circ (X^3YZ^3-X^2Y^3Z^2)=XZ-Y^2$, and for $F=X^3+Y^3+Z^3, A_F=R/I_F$ where $I_F=\Ann(F)=(xy,xz,yz,x^3-y^3,x^3-z^3)$.
 \begin{lemma}\label{maclem}  There is a 1-1 correspondence between \par
 i. Elements $F\in S_j$, up to non-zero constant multiple, and\par
 ii. AG quotients $A=R/I$ of socle-degree $j$, given by\par
  \begin{align}F&\to  A_F=R/I_F, I_F=\Ann (F), \text{ and}\notag\\
  A=R/I& \text{ of socle degree } j \to  F=I_j^\perp\in S_j.
  \end{align}
  For $\ell\in R_1$ we have $\Ann(\ell\circ F)=I_F:\ell$.
 \end{lemma}
 \begin{corollary} The rank matrix $M_{A_F,\ell}$ for an AG algebra has main diagonal the Hilbert function $H(A_F), A_F=R/I_F$ and parallel diagonals the Hilbert functions of $A_{\ell^a\circ F}\,, a=1,2,\ldots$.
 \end{corollary}
 \begin{proof} This follows from Lemma \ref{ranklem} (a), and from $\Ann(\ell\circ F)=I_F:\ell$ from Lemma \ref{maclem}.
 \end{proof}
 \noindent
 \begin{example}[Dependence of JDT on characteristic of $\sf k$]\label{MAex} For $F=X^3+Y^3+Z^3$ and $\ell=x+y+z$ the diagonals of $M_{A,\ell}$ are
 $(1,3,3,1), (1,3,1), (1,1)$ and $1$, respectively arising from $F, \ell\circ F=X^2+Y^2+Z^2, \ell^2\circ F=X+Y+Z$ and $\ell^3\circ F=(3)$ - unless the characteristic of $\sf k$ is three, in which case the last diagonal is $0$.\par
 Here the JDT of $(A,\ell)$ where $A={\sf k}[x,y,z]/I, I=\Ann F=(xy,xz,yz,x^3-y^3,x^3-z^3)$ is $\mathcal S=(4_0,2_1)$ unless the characteristic of $\sf k$ is three, in which case $\mathcal S=(3_0,3_1,1_{1,2})$.  Thus, given the algebra $A$, the possible JDT may depend upon the characteristic $p$. However, the Jordan type $\mathcal S=(j_0,j_1,1\uparrow^{j-1}_1)$ occurs for $F^\prime=X^{j-1}Y+Z^j$ and $\ell=x$, case \#6 in Table \ref{1351fig}, for any infinite field $\sf k$.  We find that the set of occurring JDT for a given Hilbert function $H$ does not depend on the characteristic, for the almost constant Gorenstein sequences $H$ here.
 \end{example}

\begin{proof}
[Proof of (iv) of Lemma \ref{ranklem}] (Correction of proof in \cite[Lemma~3.7]{Al1}).  We let $A=R/\Ann F$ where $F$ is a degree-j form in $S$, and for $\ell\in R_1$ recall that $\ell\circ F$ is the contraction of Equation \ref{contracteq}.  Following \cite{Al1}, the notation $\langle \ell^{i}\circ F\rangle$ stands for the Macaulay inverse system module $\langle \ell^{i}\circ F\rangle= R\circ (\ell^{i}\circ F)$ of the algebra $R/\Ann(\ell^{i}\circ F)$. 
The surjective map $\circ \ell: \langle \ell^{i}\circ F\rangle\twoheadrightarrow \langle \ell^{i+1}\circ F\rangle$ for every $i\geq 0$ induces the following commutative diagram
$$
\xymatrix{
0\ar[r]& \langle \ell^{i+1}\circ F\rangle \ar[r]&\langle \ell^{i}\circ F\rangle  \ar[r]& \langle \ell^{i}\circ F\rangle/ \langle \ell^{i+1}\circ F\rangle \ar[r]&0\\
0\ar[r]& \langle \ell^i\circ F\rangle\ar[r]\ar[u]^{\circ \ell}&\langle \ell^{i-1}\circ F\rangle\ar[r]\ar[u]^{\circ\ell}& \langle \ell^{i-1}\circ F\rangle/ \langle \ell^i\circ F\rangle\ar[r]\ar[u]^{\varphi}&0\\
}
$$
This shows that $\varphi$ is surjective, which implies Equation \eqref{sumdiageq} and Lemma \ref{ranklem}(iv).
\end{proof}

\subsubsection{Map $\Gor(T)$ to the punctual Hilbert scheme.}
 When the Sperner number occurs at least 3 times; i.e. $k\geq 3$ there is a relation between the family $\Gor(T)$ and the punctual Hilbert scheme $\Hilb^s(\mathbb P^2)$, that parametrizes length-$s$ subschemes of the projective plane $\mathbb P^2$.  We assume for all parametrization results that $\sf k$ is an infinite field.  Given a codimension three Gorenstein sequence  $T=(1,3,T_1,s^k,T_1^r,3,1)$ as in Equation \eqref{Gorseqeqn}, we denote by $H_T=(1,3,T_1,\overline{s})=(1,3,T_1,s,s,\ldots)$, and note that $T=\Sym(H_T,j)$, the symmetrization of $H_T$ around socle degree $j/2$:
 \begin{equation}\label{SymHeq}
\qquad\qquad Sym(H,j)_i=\begin{cases} &H_i \text { if } 0\le i\le j/2;\qquad\qquad\qquad\qquad\qquad\qquad\qquad\qquad\qquad\\
    &H_{j-i} \text { if } j/2\le i\le j.\qquad\qquad\qquad\qquad\qquad\qquad\qquad\qquad\qquad\qquad\qquad\qquad.
 \end{cases}
 \end{equation}
 The following is from \cite[Theorem 5.31]{IK}.
 Recall that an annihilating scheme $\Z$ for a form $G\in S_j, S={\sf k}[X,Y,Z]$ is a scheme such that $I_Z\subset I_G=\Ann G$; it is \emph{tight} if the length $|\Z|=$ the Sperner number of $T=H(R/\Ann G)$.
 \begin{lemma}[Map $\Gor(T)$ to  $\Hilb^s(\mathbb P^2)$]\label{maphilblem} Let $T=(1,3,T_1,s^k,T_1^r,3,1)$ as in Equation \eqref{Gorseqeqn}, let $k\ge 3$ and assume $A=R/I\in \Gor(T)$. There is a length-$s$ scheme $\Z\subset \mathbb P^2$ of regularity degree $\tau=\min\{i\mid T_i=s\}$, and Hilbert function $H(R/I_\Z)=H_T$, such that $I_\Z=(I_{\le \tau +2})$.  The scheme $\Z$ is the unique tight annihilating scheme for $A$.\par
 Conversely, let the punctual scheme $\Z \subset \mathbb P^2$ have regularity degree $\tau$, and let $j\ge 2\tau+2$. Then for a general enough $G\in (I_Z)_j^\perp\in S_j$ we have $T=H(R/\Ann G)=\Sym (H_\Z,j)$, and $\Z$ is a tight annihilating scheme for $G$.
 \end{lemma}
 \begin{remark} M. Boij \cite[Theorem 4.2]{Bo2} has shown that there may be arbitrarily long constant subequences
 of a Gorenstein sequence in codimension four or more, without $A\in \Gor(T)$ having necessarily a tight punctual annihilating scheme. Thus, the conclusion of Lemma \ref{maphilblem} is confined to codimension three (see also \cite[Example 6.43ff]{IK}).\par
 It is known that a tight annihilating scheme for $F$ may be chosen to be Gorenstein: see \cite[Proposition 4.3]{BOT}.

 \end{remark}
 The following Lemma is for arbitrary codimension three Gorenstein sequences; for the special case of almost constant $T$ see 
 Corollary \ref{Mcor} and Equation~\eqref{Meq}.
\begin{lemma}\label{diaglem} Fix a codimension three Gorenstein sequence $T$ as in Equation \eqref{Gorseqeqn} with Sperner number (maximum value) $s$ that occurs at least 3 times, $t_a=t_{a+1}=\cdots =t_b=s$. Fix $A\in \Gor(T)$ and fix any linear form $\ell\in A_1$.\par
Then the rank of $m_{\ell^k}: A_i\to A_{i+k}$ is independent of $i$ if  $a\le i\le b-k$. Equivalently, in the rank matrix each diagonal of $M_{A,\ell}$ parallel to the main diagonal is constant within the square corresponding to those degrees $i$ for which $t_i=s$.\par
\end{lemma}
\begin{proof} Let $A=R/I$. By Lemma \ref{maphilblem} there is a scheme $\Z\subset \mathbb P^2$ such that $(I_\Z)_i=I_i$ for $a\le i\le b$. Fix a non-zero-divisor $\alpha\in R/I_\Z$.  That $\ell^k\alpha=\alpha\ell^k$ acting on $R/I_\Z$ for $i\ge a$ shows the desired equality of ranks.
\end{proof}
Assuming $k\geq 3$, the Lemma \ref{diaglem} implies that the entries of $M_{A,\ell}$ satisfy identities
\begin{align*}
    &r_1 : = r_{a,j-a-1}=r_{a+1,j-a-2}=\cdots =r_{\lfloor j/2 \rfloor,j-\lfloor j/2\rfloor-1}\\
    &r_2 : = r_{a,j-a-2}=r_{a+1,j-a-3}=\cdots =r_{\lfloor j/2\rfloor,j-\lfloor j/2\rfloor-2}\\
    &\vdots\\
    &r_{k-1} : = r_{a,k-1}.
\end{align*}
Note that $r_1\ge 1$ since otherwise $\ell\in I$.
\noindent By Lemma \ref{ranklem} (iv), we have
\begin{equation}\label{diffeq}
2r_i\geq r_{i-1}+r_{i+1} \text { for all } i=1,\dots, k-2.
\end{equation}
We define the following sequences of non-negative integers (taking $r_0=s$)\par
 \begin{align}
\overrightarrow{r}:&=(r_1,\ldots  ,r_{k-1});\notag\\
 \overrightarrow{\Delta}:&=\{\Delta r_i, i=1,\dots, k-1\},\quad\Delta r_i=r_{i-1}-r_{i}.\label{deltaeq}
 \end{align}
We let $p(u,t)$ denote the number of partitions of $u$ whose largest part is less or equal $t$, and we let $p(u)$ denote the number of partitions of $u$.  We will sometimes denote $\overrightarrow{\Delta}$ by $\Delta r$ or $\Delta$.

\begin{corollary}[The $\Delta$ sequence]\label{deltacor}  Assume that $T=(1,3,T_1,s^k,T_1^r,3,1)$ with $k\ge 3$ is a codimension three Gorenstein sequence as in Equation \eqref{Gorseqeqn}. We have the following:
		\begin{enumerate}[(a)]
			\item the sequence $\overrightarrow{\Delta}$ is non-increasing;
			\item $\Delta r_i\le s-1$ for all $i$.
			\item $\Sigma_{i=1}^{k-1} \,\Delta r_i =s-r_{k-1}\le s$;
			\item $\Delta r_i=0$ for all $i\ge  r_0+1=s+1$, so $r_{s}=r_{s+1}=\cdots = r_{k-1}$ when $k\ge s+1$. 
		\end{enumerate}
\end{corollary}

\begin{proof} Item (a)
	is equivalent to Equation \eqref{diffeq}.  Parts (b),(c) follow from Lemma \ref{ranklem}: Part (b) $\Delta r_i\le s-1$ because $r_1\ge 1$ and $\overrightarrow\Delta$ is non-increasing; and Part (c) because we have $\Sigma_i \,\Delta r_i + r_{k-1}=r_0$.  From (a),(b) and (c), we conclude that the sequence $\overrightarrow{\Delta}$ is a partition of some integer $r\le s$, which proves part (d).
\end{proof}

Recall that the almost constant $T$ of  Equation \eqref{Taceq} - satisfy $
T = (1,3,s^k,3,1) \text { with } 3\leq s \le  6 \text { and }  k\ge 1.$  We have
\begin{corollary}[The $\Delta$ sequence for an almost constant Hilbert function]\label{Mcor} Assume $T$ satisfies \eqref{Taceq}. Then the rank matrix $M_{A,\ell}$ of any pair $(A,\ell)$ where $A\in \Gor(T)$ and $\ell\in \mathfrak{m}$ has the following form

\begin{equation}\label{Meq}
M_{A,\ell}=\begin{pmatrix}
\color{blue}1&\color{blue}r_{0,j-1}&\color{blue}r_{0,j-2}&\color{blue}r_{0,j-3}&\color{blue}r_{0,j-4}&\color{blue}r_{0,j-5}&\color{blue}\dots &\color{blue}r_{0,2}&\color{blue}r_{0,1}&\color{blue}r_{0,0}\\\
&\color{blue}3&\color{blue}r_{1,j-2}&\color{blue}r_{1,j-3}&\color{blue}r_{1,j-4}&\color{blue}r_{1,j-5}&\color{blue}\dots &\color{blue}r_{1,2} &\color{blue}r_{1,1}&\color{blue}r_{0,1}\\\
&&s&r_1&r_2&r_3&\dots & r_{k-1}&\textcolor{blue}{r_{1,2}}&\textcolor{blue}{r_{0,2}}\\
&&&s&r_1&r_2&\dots & r_{k-2}&\textcolor{blue}{r_{1,3}}&\textcolor{blue}{r_{0,3}}\\
&&&&s&r_1&\dots & r_{k-3}&\textcolor{blue}{r_{1,4}}&\textcolor{blue}{r_{0,4}}\\
&&&&&\ddots&\ddots&\vdots &\textcolor{blue}{\vdots}&\textcolor{blue}{\vdots}\\
&&&&&&\ddots&r_1&\textcolor{blue}{r_{1,j-3}}&\textcolor{blue}{r_{0,j-3}}\\
&&&&& &&s&\textcolor{blue}{r_{1,j-2}}&\textcolor{blue}{r_{0,j-2}}\\
&&&&& &&&\textcolor{blue}{3}&\textcolor{blue}{r_{0,j-1}}\\
&&&&& &&&&\textcolor{blue}{1}\\
\end{pmatrix},
\end{equation}
where the entries below the main diagonal are all zero. Because of the connection of such an AG algebra with a punctual Hilbert scheme (Lemma~\ref{diaglem}), the $k\times k$ upper triangular matrix in the center (colored black) is a Toeplitz matrix whose main diagonal is $s$ and whose constant diagonals parallel to the main diagonal are $r_1,\ldots, r_{k-1}$.
\end{corollary}

\begin{remark}\label{deltarmk}  Note that if $T$ satisfies Equation \ref{Taceq} for $k$ at least $3$, then $\Delta r_1\le 3$ since otherwise the first difference of the diagonals in the rank matrix would not be an $O$-sequence. Take for example $s=6,\Delta r_1=4$: then the diagaonal $H(A)$ and $H(A(1))_+$ of $M_{A,\ell}$ are\par\vskip 0.2cm
	\qquad $\begin{array}{c|ccccc}
		+&1&3&6&6&\ldots\\
		-&0&1&2&2&\ldots\\\hline
		&1&2&4&4&\ldots\\
	\end{array},
	$\par\noindent
	whose difference $(1,2,4,4,\ldots)$ is not an O-sequence, as required by Lemma \ref{ranklem}(ii).
	
\end{remark}\subsubsection{Jordan Degree Type matrix.}\label{JDTmatrixsec}

\begin{definition}\label{def:jordan matrix}
  Suppose that $M_{A,\ell}$ is the rank matrix of an AG algebra $A$ and the linear form $\ell$. Then
the \emph{Jordan degree type matrix (JDT matrix)}, $J_{A,\ell}$, of the pair $(A,\ell)$ is defined to be the upper triangular matrix with the following entries for $u\le v$
\begin{align}\label{eq:jordan matrix}
(J_{A,\ell})_{u,v} :
=& (M_{A,\ell})_{u,v}+(M_{A,\ell})_{u-1,v+1}-(M_{A,\ell})_{u-1,v}-(M_{A,\ell})_{u,v+1},
\end{align}
where we set $(M_{A,\ell})_{u,v}=0$ if either  $u< 0$ or $v< 0$.  We may write this as $J_{A,\ell}=J_M$, as it depends only on the matrix $M$.
\end{definition}
Lemma \ref{ranklem}(iv.) implies that the entries of $J_{A,\ell}$ are all non-negative. For every $0\le u\le v\le j$ if $J_{u,v}\neq 0$ then in the Jordan degree type $\mathcal{S}_{A,\ell}$ the multiplicity of the pair $(p,\nu)$ where $p=v-u+1$ and $\nu=u$ is $\eta(p,\nu)=J_{u,v}$. See \cite{Al1, Al2} for more detail. \par
\begin{remark}[Relation between JDT matrix ($J_{A,\ell}$) and JDT ($\mathcal{S}_{A,\ell}$)]\label{JDTmatrixtoJDTrem} The row of the JDT matrix gives the
index, which is the initial degree of the corresponding entry of $\mathcal S_{A,\ell}$. The diagonal (with main diagonal counted as 1) gives the value, so the $w$-th diagonal parallel to the main diagonal gives an entry $w$ of $\mathcal S$ with subscript the row (first row counted as zero).\par
\end{remark}
\par

\begin{example}\label{rankex}
  Let  $R={\sf k}[x_1,x_2,x_3], A=R/I, I=\Ann(X_1^2X_2X_3)=(x_1^3,x_2^2,x_3^2)$ and $\ell=x_1$. The Hilbert function of $A$ is $T = (1,3,4,3,1)$. The following are the rank matrix and JDT matrix for $(A,\ell)$
  \begin{equation}\label{array1eq}
M_{A,\ell}=\left(\begin{array}{ccccc}
1&1&1&0&0\\
0&3&3&2&0\\
0&0&4&3&1\\
0&0&0&3&1\\
0&0&0&0&1
\end{array}\right).\qquad 
J_{A,\ell}=\left(\begin{array}{ccccc}
0&0&1&0&0\\
0&0&0&2&0\\
0&0&0&0&1\\
0&0&0&0&0\\
0&0&0&0&0
\end{array}\right).
\end{equation}
The JDT for $\ell=x_1$ and $A$ is $\mathcal{S}_{A,\ell}=(3_0,3_1^2,3_2)$, which can be seen from $A\cong {\sf k}[x_1]/(x_1^3)\otimes {\sf k}[x_2,x_3]/(x_2^2,x_3^2)$, so $\mathcal S_{A,\ell}=(3_0)\otimes (1_0,1_1^2, 1^2_2)$.  Or we can recover $\mathcal S_{A,\ell}$ directly from $M_{A,\ell}$ working consecutively, from larger to smaller $a$, since $T(a):=T(R/(I:x_1^a))$ is the $a$-th diagonal above and parallel to the main diagonal $T(A)$ of $M_{A,x_1}$ \cite[Proposition 3.4]{Al1}. So we have from Equation \eqref{array1eq} that  $T(3)=T(R/(I:x_1^3))=0$ implying $x_1^3\in I$.
The diagonal $(1,2,1)= T(R/(I:x_1^2))$; since $x_1^3\in I$ this implies
$\mathcal S_{R/(I:x_1^2),x_1}=(1_0,1_1^2,1_2)$. Then $T(R/(I:x_1))=(1,3,3,1)$ implies $\mathcal S_{R/(I:x_1),x_1}=(2_0,2_1^2,2_2)$, and we conclude from $T(A)=(1,3,4,3,1)$ that $S_{A,x_1}=(3_0,3_1^2,3_2)$. 
\noindent An alternative approach is to determine $\mathcal{S}_{A,\ell}$ from $J_{A,\ell}$ using Definition \ref{def:jordan matrix} and the discussion after it. By Equation~\eqref{eq:jordan matrix} we get that the only non-zero entries of $J_{A,\ell}$ are $(J_{A,\ell})_{0,2}=1$,$(J_{A,\ell})_{1,3}=2$, and $(J_{A,\ell})_{2,4}=1$. Therefore, $\mathcal{S}_{A,\ell}=(3_0,3_1^2,3_2)$. 
\end{example}

 An example of constructing an algebra from a particular rank matrix for $T=(1,3,4,3,1)$ is found in \cite[Example 4.1]{Al3}. See also Section \ref{exsec} below.\par 
As in the Example \ref{rankex} the rank matrix of a graded AG algebra determines the Jordan degree type, and vice-versa \cite[Proposition 3.12]{Al1}. That the rank matrix determines the JDT is essentially the result that the ranks of mixed Hessians determine JDT (see \cite{GoZa,CGo}): this rests on the work of T.~Harima and J. Watanabe on central simple modules \cite{HW}.  

 \section{Potential Jordan degree types.}\label{potentialsec}
 In this section we assume that $\sf k$ is an infinite field of arbitrary characteristic. We denote by $\mathcal J(T)$ the set of Jordan degree types (JDT) possible for pairs $(A,\ell), A\in \Gor(T), \ell\in A_1$. We recall the Jordan degree types for complete intersections of codimension two, then in codimension three determine the potential Jordan types for almost constant Hilbert functions of each Sperner number  $3,4,5,6$ (Section \ref{cod3sec}).

 \subsection{Codimension two.}\label{cod2sec}
The number of possible rank matrices, equivalently JDT, for codimension two complete intersection (CI) Hilbert functions was obtained  explicitly in \cite{AIK}. We first state a portion of that result - giving the count. We then in an Example list the rank matrices in codimension two corresponding to Sperner numbers $s=2,3$. 

Recall that the Hilbert function of a codimension two complete intersection of Sperner number $d$ and socle degree $j=k+2d-2$ satisfies
\begin{equation}\label{CIHF}
T=(1,2,\ldots,d-1,d^k, d-1,\ldots, 2, 1) \text { where } k\ge 1.
\end{equation}

For such $T$, the automorphism $\iota$ on $\mathcal J(T)$ flips (switches rows and columns) in the smallest rectangle in the Jordan type, and makes the corresponding change for the JDT (see \cite[Theorem 3.23]{AIK}). For example, the Jordan degree type $\mathcal{S}=((j+1)_0, (j-1)_1,(j-1)_2)$ satisfies $\iota(\mathcal{S})=((j+1)_0, 2\uparrow_1^{j-1})$.

The following is adapted from \cite[Theorems 2.21, 2.35]{AIK}, which have further very detailed combinatorial data. 
\begin{theorem}\cite{AIK}\label{cod2lem}
Let $T$ satisfy Equation \eqref{CIHF}.  When $k=1$ there are $2^{d-1}$ Jordan types possible for $A\in CI(T)$, all weak-Lefschetz - with $d$ parts. When $k\ge 2$ there are $2^d$ Jordan types possible for $A\in CI(T)$: half are weak Lefschetz - with $d$ parts, and half, with $d+k-1$ parts  are of the form $\iota(\mathcal S)$.\par
When $k=1$, the smallest part $\alpha$ of $\mathcal S$ occurs $\alpha$ times, so $\iota(\mathcal S)=\mathcal S$.

\end{theorem}
\begin{example}\label{cod2ex}
In Tables \ref{table:122} and \ref{table:123}, respectively, we list the possible rank matrices for codimension two, almost constant CI sequences of Sperner numbers 2 and 3, respectively. \par
Note that in Table  \ref{table:122}, the first row $\Delta=0$ corresponds to the two
JDT $((j+1)_0,(j-1)_1)$ and $(j_0,j_1)$. The second row corresponds to $\iota(((j+1)_0,(j-1)_1)=(j+1)_0, 1\uparrow_1^{j-1})$ and the third row to  $\iota((j_0,j_1))=(2\uparrow_0^{j-1}).$
\end{example}
Another aspect of this case, the role of $\Delta$ in determining repeated parts of $\mathcal S_{A,\ell}$, is shown in Proposition \ref{partsprop} and Example \ref{reppartex}(c) below.

\begin{table}
 \begin{center}
    \begin{tabular}{|c|c|c|c|c|}
    \hline
     $\#$& $\overrightarrow{\Delta}  $   &  $\overrightarrow r $&\# rk matrices &JDT\\
        \hline
        $1_a$& $(0,0)$&$(2,2,\dots,2)$&2&$((j+1)_0, (j-1)_1) $\\
        $1_b$&&&&$(j_0,j_1) $\\\hline
        $2_a$& $(1,0)$&$(2,1,\dots,1)$&1&$((j+1)_0, 1\uparrow_1^{j-1})$\\\hline
         $3_b$& $(1,1)$&$(2,1,0,\dots,0)$&1&$(2\uparrow_0^{j-1})$\\
         \hline 
    \end{tabular}
\end{center}
\caption{Case $s=2, k\ge 2;$ codimension two: the action of $\iota$ takes $1_b$ to $3_b$.}\label{table:122}
\end{table}
\begin{table}
 \begin{center} 
    \begin{tabular}{|c|c|c|c|c|}
    \hline
     $\#$& $\overrightarrow{\Delta} $   &  $\overrightarrow r $&\# rk matrices&JDT\\
        \hline
        $1_a$& $(0,0,0)$&$(3,3,3,3\dots,3)$&4&$((j+1)_0, (j-1)_1, (j-3)_2) $\\
        $1_b$&&&&$(j_0,j_1,(j-3)_2)$\\
        $1_c$&&&&$((j+1)_0,(j-2)_{1,2})$\\
        $1_d$&&&&$((j-2)\uparrow_0^2)$\\\hline
        $2_a$& $(1,0,0)$&$(3,2,2,2,\dots,2)$&2&$((j+1)_0, (j-1)_1, 1\uparrow_2^{j-2}) $\\
       $2_b$&&&&$(j_0,j_1,1\uparrow_2^{j-2})$\\\hline
      $3_c$& $(1,1,0)$&$(3,2,1,1,\dots,1)$&1&$((j+1)_0,2\uparrow_1^{j-2})$\\\hline
       $4_d$& $(1,1,1)$&$(3,2,1,0,\dots,0)$&1&$3\uparrow_0^{j-2}$\\
         \hline 
    \end{tabular}
\end{center}
\caption{Case $s=3, k\ge 2$, codimension 2, $T=(1,2,3^k,2,1)$. The action of $\iota$ takes $1_a$ to $2_a$. See \cite[Figure 10]{AIK}.}\label{table:123}\par\vskip 0.2cm
 
\end{table}

\subsection{Codimension three. }\label{cod3sec}
In this section we count the number of potential rank matrices and Jordan degree types for codimension three AG algebras with almost constant Hilbert functions. More precisely, for each Sperner number $3\le s\le 6$ we count all possible rank matrices for each possible sequence $\overrightarrow \Delta$ of Equation \eqref{deltaeq} satisfying the conditions in Lemma \ref{ranklem} and Corollary \ref{deltacor}. 
 \subsubsection{Case $s=3$.}
 Let $T=(1,3^k,1)$, $k\ge 3$. 
 Here we assume that $T$ is the Hilbert function of codimension three AG algebras. \par In this case the rank matrix has the following form.
\begin{equation}\label{eq: s=3 matrix}
M=\begin{pmatrix}
\color{blue}1&\color{blue}r_{0,j-1}&\color{blue}r_{0,j-2}&\color{blue}r_{0,j-3}&\color{blue}r_{0,j-4}&\color{blue}\dots &\color{blue}r_{0,1}&\color{blue}r_{0,0}\\\
&3&r_1&r_2&r_3&\dots & r_{k-1}&\textcolor{blue}{r_{0,1}}\\
&&3&r_1&r_2&\dots & r_{k-2}&\textcolor{blue}{r_{0,2}}\\
&&&3&r_1&\dots & r_{k-3}&\textcolor{blue}{r_{0,3}}\\
&&&&\ddots&\ddots&\vdots &\textcolor{blue}{\vdots}\\
&&&&&\ddots&r_1&\textcolor{blue}{r_{0,j-2}}\\
&&&& &&3&\textcolor{blue}{r_{0,j-1}}\\
&&&& &&&\textcolor{blue}{1}\\
\end{pmatrix}.
\end{equation} 
\begin{theorem}\label{s=3thm}
     Let $T = (1,3^k,1)$  for $k\ge 3$. There are 8 potential Jordan degree types for pairs $(A,\ell)$, where $A\in \Gor(T)$ and $\ell\in A_1.$
     \
 \end{theorem}
 \begin{proof}
     First, we count possible sequences $\overrightarrow\Delta$, obtaining $p(0)+p(1)+p(2)+p(3,2)=1+1+2+2=6$.  We list them in Table \ref{s=3tab}. 
     Every rank matrix is determined by the sequence $\overrightarrow r$ (equivalently $\overrightarrow \Delta$ and entries $r_{0,0},\dots, r_{0,j-1}$ in the matrix given in Equation \eqref{eq: s=3 matrix}. But since each diagonal of the rank matrix is a Gorenstein sequence the entries $r_{0,2},r_{0,3}, \dots r_{0,j-1}$ are determined by the sequence $\overrightarrow{r}$. Therefore, for a given sequence $\overrightarrow r$ and values $r_{0,1}$ and $r_{0,0}$ the rank matrix is determined. Observe that $r_{0,1}$ and $r_{0,0}$ are the entries of the top right $2\times 2$ submatrix of $M$ as in Equation \eqref{eq: s=3 matrix}. Thus to find the possible rank matrices in this case we find the possible sequences $\overrightarrow{r}$ and for each of these sequences we find the possible submatrices $\begin{pmatrix}r_{0,1}&r_{0,0}\\r_{k-1}&r_{0,1}
 \end{pmatrix}$.\par
     For $\overrightarrow\Delta = (0,0,0)$, we have $\overrightarrow r=(3,3,\dots ,3)$. So in particular $r_{k-1}=3$ and there are two possible choices for the top right corner of the rank matrix: $$\begin{pmatrix}
         1&1\\3&1
     \end{pmatrix}, \begin{pmatrix}
         1&0\\3&1
     \end{pmatrix}.$$\par
     For $\overrightarrow\Delta = (1,0,0),(1,1,0),$ and $(1,1,1)$ the second diagonal of the rank matrix is $(1,2^{k-1},1)$ that is a codimension two Gorenstein Hilbert function. Therefore, the number of possible rank matrices can be obtained from the corresponding row of Table \ref{table:122}.\par
For $\overrightarrow\Delta = (2,0,0)$ we have $r_{k-1}=1$ and this forces $r_{0,1}=r_{0,0}=1$. Also, clearly for $\overrightarrow\Delta = (2,1,0)$ we have $r_{0,1}=r_{0,0}=0$ forced by $r_{k-1}=0.$
 \end{proof}
 \begin{table}[h]
 \begin{center}
    \begin{tabular}{|c|c|c|c|c|c|}
    \hline
     $\#$& ${ \overrightarrow{\Delta}} $   &  $\overrightarrow{r}=(r_0, \dots , r_{k-1})$&\# rk matrices &JDT&Table \ref{1351fig}\\
        \hline
        1a& $(0,0,0)$&$(3,3,3,3\dots,3)$&2&$((j+1)_0, ((j-1)_1)^2)$&\#1\\
         1b&&&&$(j_0, j_1,(j-1)_1)$ &\#5\\\hline
        2a& $(1,0,0)$&$(3,2,2,2,\dots,2)$&2&$((j+1)_0, (j-1)_1,1\uparrow_1^{j-1})$&\#2\\
        2b&&&&$(j_0,j_1,1\uparrow_1^{j-1})$&\#6\\\hline
        3& $(2,0,0)$&$(3,1,1,1,\dots,1)$&1&$((j+1)_0,(1\uparrow^{j-1}_1)^2)$&\#4\\\hline
        4& $(1,1,0)$&$(3,2,1,1,\dots,1)$&1&$((j+1)_0, 2\uparrow_1^{j-2}, 1_{1,j-1})$&\#3\\\hline
        5& $(2,1,0)$&$(3,1,0,0,\dots,0)$&1&$(2\uparrow_0^{j-1},1\uparrow_1^{j-1})$&\#8\\\hline
        6& $(1,1,1)$&$(3,2,1,0,\dots,0)$&1&$(3\uparrow_0^{j-2},1_{1,j-1})$&\#7\\
         \hline 
    \end{tabular}
\end{center}
\caption{Case $s=3$, $T=(1,3^k,1)$, $k\ge 3$, codimension three; list of JDTs and reference to corresponding entry of Table \ref{1351fig}.} \label{s=3tab}
\end{table}
\normalsize

\begin{example}\label{JDTex}
\begin{enumerate}
    \item [(i)] When $s=3,j=4$ and $\overrightarrow{\Delta}=(0,0,0)$ then $\overrightarrow{r}=(3,3,3)$ and the upper right corner of the rank matrix $M$ is one of the following
 (they correspond respectively to \#1a, \#1b in Table~\ref{s=3tab}, for $j=4$)\vskip 0.2cm
 \par
 \qquad$\begin{pmatrix}1&\color{red}1\\\color{red}3&1
 \end{pmatrix}$  \,  $\mathcal{S}=(5_0, (3_1)^2)$; \quad or 
 $\begin{pmatrix}{\color{red}1}&0\\\color{red}{3}&\color{red}{1}
 \end{pmatrix}$\,  $\mathcal{S}=(4_0, 4_1,3_1)$.\par
\item[(ii)]  When $s=3,j=4$ and $\overrightarrow{\Delta}=(1,0,0)$ then $\overrightarrow{r}=(3,2,2,\ldots)$ and the upper right corner of $M$ is one of the following
 (we give the JDT for $j=4$ here - they are, respectively, \#2a, \#2b in Table~\ref{s=3tab})\vskip 0.2cm
 \par
 $\begin{pmatrix}1&\color{red}1\\\color{red}2&1
 \end{pmatrix}$  \, $\mathcal{S}=((5)_0, (3)_1,1\uparrow_1^3)$; \quad or 
 $\begin{pmatrix}{\color{red}1}&0\\2&\color{red}{1}
 \end{pmatrix}$\,  $\mathcal{S}=(4_0, 4_1,1\uparrow_1^3)$.\vskip 0.2cm
\item[(iii)] When $s=3, j=4$ and $\overrightarrow{\Delta}=(2,0,0)$, then $\overrightarrow{r}=(3,1,1)$ we have\par\vskip 0.2cm
M=$\begin{pmatrix} 1&1&1&1&\color{red}1\\
&\color{red}3&1&1&1\\
&&\color{red}3&1&1\\
&&&\color{red}3&1\\
&&&&1
\end{pmatrix},$\quad  
$J=\begin{pmatrix} 0&0&0&0&1\\
&2&0&0&0\\
&&2&0&0\\
&&&2&0\\
&&&&0
\end{pmatrix},$ \,  $\mathcal{S}=(5_0, (1\uparrow_1^3)^2)$,
 \par\vskip 0.2cm\noindent
 as in the row 3 of Table \ref{s=3tab} for $j=4$. The JDT can be read off simply from the JDT matrix $J$ of Definition \ref{def:jordan matrix}, as in Remark \ref{JDTmatrixtoJDTrem}.
 \end{enumerate}
  \end{example}
\subsubsection{Case $s=4$.}
Let $T=(1,3,4^k,3,1)$, $k\ge 3$.
The number of possible sequences $\overrightarrow{\Delta}$ (so $\overrightarrow{r}$) is equal to $$p(0)+p(1)+p(2)+p(3)+p(4,3)=1+1+2+3+(5-1)=11.$$ In the following theorem, we find the number of rank matrices, which are of the form \eqref{Meq}, for each of those sequences and list them in Table \ref{s4tab} below.
 \begin{theorem}\label{4kthm}
    Let $T=(1,3,4^k,3,1)$ for $k\ge 3$. 
    There are 26 potential Jordan degree types for $(A,\ell)$ where  $A\in \mathrm{Gor}(T)$ and $\ell\in A_1$. 
\end{theorem}
\begin{proof}
To get the possible rank matrices for $\overrightarrow\Delta = (0,0,0,0)$
we note that since the diagonals are Hilbert functions for AG algebras, we have $r_{1,j-2}=\cdots=r_{1,3}=3$, and $r_{0,4}=\cdots=r_{0,j-1}=1$. It is left to count the possibilities of the matrix\vskip 0.2cm\par $\begin{pmatrix}
    1&r_{0,3}&r_{0,2}&r_{0,1}&r_{0,0}\\
    3&3&r_{1,2}&r_{1,1}&r_{0,1}\\
    4&4&4&r_{1,2}&r_{0,2}\\
    4&4&4&3&r_{0,3}\\
    4&4&4&3&1    
\end{pmatrix}.$\vskip 0.2cm\par\noindent
By Lemma \ref{diaglem} (iv), $r_{1,2}=3$ and $r_{0,3}=r_{0,2}=1$, and we have the following 5 choices for the upper right corner:  $$\begin{pmatrix}
    1&0&0\\
    3&2&0\\
    4&3&1
\end{pmatrix},\begin{pmatrix}
    1&1&0\\
    3&2&1\\
    4&3&1
\end{pmatrix},\begin{pmatrix}
    1&1&1\\
    3&2&1\\
    4&3&1
\end{pmatrix},\begin{pmatrix}
    1&1&0\\
    3&3&1\\
    4&3&1
\end{pmatrix},\begin{pmatrix}
    1&1&1\\
    3&3&1\\
    4&3&1
\end{pmatrix}.$$
Now suppose the first entry of $\overrightarrow\Delta$ is one, or equivalently $r_1=3$, so we have either $r_{1,j-2}=2$ or $3$. If $r_{1,j-2}$ is 3, the number of possible rank matrices follows from Table \ref{s=3tab}. If $r_{1,j-2}=2$, the number of possible rank matrices follows from the case $T(1)=(1,2,3^{k-1},2,1)$, (see Table \ref{table:123}). For instance, in $\#2$ in Table \ref{s4tab}, the number of rank matrices is 6 such that 2 of them correspond to having  $T(1)=(1,3^{k+1},1)$ for which we look at $\#1$ in Table \ref{s=3tab} that counts 2 possible rank matrices. Similarly, if $T(1)=(1,2,3^{k-1},2,1)$ the 4 rank matrices are obtained from $\#1$ in Table \ref{table:123}.
These two sequences contribute to cases \#2, 4, 7 and 11 (when $r_1=3).$\par
When $r_1=2$, $r_{1,j-2}$ is also equal to $2$, and the number of possible rank matrices follows from Table \ref{table:122}.\par
When $r_{k-1}=1$ (respectively $r_{k-1}=0$), the top right $3\times 3$-submatrix of $M_{A,\ell}$ is forced to be $\begin{pmatrix}
		1&1&1\\
		1&1&1\\
		1&1&1
	\end{pmatrix}$, respectively $\begin{pmatrix}
	0&0&0\\
	0&0&0\\
	0&0&0
	\end{pmatrix}$.
\end{proof}
 \begin{table}[h]
    \begin{center} 
    \begin{tabular}{|c|c|c|c|c|}
    \hline
     \#& $\overrightarrow \Delta=(\Delta r_1,\dots,\Delta r_4) $   &  $\overrightarrow r=(r_0,r_1,\dots,r_{k-1})$&\multicolumn{2}{c|}{\#rk matrices}\\
        \hline
         1&$(0,0,0,0)$&$(4,4,4,4,4,\dots,4)$&\textcolor{white}{aa}5\textcolor{white}{aa}&0\\
         2&$(1,0,0,0)$&$(4,3,3,3,3,\dots,3)$&2&4\\
         3&$(2,0,0,0)$&$(4,2,2,2,2,\dots,2)$&2&0\\
         4&$(1,1,0,0)$&$(4,3,2,2,2,\dots,2)$&2&2\\
         5&$(3,0,0,0)$&$(4,1,1,1,1,\dots,1)$&1&0\\
         6&$(2,1,0,0)$&$(4,2,1,1,1,\dots,1)$&1&0\\
         7&$(1,1,1,0)$&$(4,3,2,1,1,\dots,1)$&1&1\\
         8&$(3,1,0,0)$&$(4,1,0,0,0,\dots,0)$&1&0\\
         9&$(2,2,0,0)$&$(4,2,0,0,0,\dots,0)$&1&0\\
         10&$(2,1,1,0)$&$(4,2,1,0,0,\dots,0)$&1&0\\
         11&$(1,1,1,1)$&$(4,3,2,1,0,\dots,0)$&1&1\\
         \hline 
    \end{tabular}
\end{center}
\caption{Case $s=4$: the sequences ${\Delta}$ and number of rank matrices. See Theorem \ref{4kthm} and Remark \ref{rankmatrixrem}.} \label{s4tab}\par  
\end{table}\normalsize
\subsubsection{Case $s=5$.}
Let $T=(1,3,5^k,3,1), k\ge 3$.  The number of possible sequences $\overrightarrow{\Delta}$ (equivalently $\overrightarrow{r}$) is equal to 
 $p(0)+p(1)+p(2)+p(3)+p(4,3)+p(5,3)=16$.
For each of the sequences $\overrightarrow \Delta$ and equivalently $\overrightarrow r$ such that $r_1$ (second entry of $\overrightarrow r$) is less than $5$ we count the number of potential rank matrices recursively using previous results for $s=3,4$ and for sequences $(r_1,r_2,\dots ,r_{k-1})$; i.e. removing $r_0$ from $\overrightarrow r$, and the appropriate $T$ for which the Sperner number occurs $k-1$ times. Thus we have the following for $s=5$ and $k\geq 4$.
\begin{theorem}\label{5kthm}
    Let $T=(1,3,5^k,3,1)$ for $k\ge 4$. 
    There are 47 potential Jordan degree types for $(A,\ell)$ where  $A\in \mathrm{Gor}(T)$ and $\ell\in A_1$. 
\end{theorem}
\begin{proof}
    The potential sequences for $\overrightarrow{\Delta}$ and $\overrightarrow{r}$
     are listed in Table \ref{s5tab}. Since by assumption  $r_0=5$ we must have $0\le \Delta r_0 = r_0-r_1<4$. If $\Delta r_0 = r_0$ we have $r_1=0$ which is equivalent to having the AG algebra $A/(0:\ell)=0$ which is impossible for $\ell\neq 0$. Also if we have $\Delta r_0 =4$ when $r_0=5$ we have $r_1=1$ which means that all the entries on the second diagonal of the rank matrix $M_{A,\ell}$, which gives the Hilbert function of $A/(0:\ell)$, are equal to one. By Lemma \ref{ranklem} part (ii), the difference between any two consecutive diagonals of the rank matrix must be an O-sequence but $(1,3,5^k,3,1)-(0,1^{k+3})=(1,2,4,4,\dots )$ clearly is not an O-sequence.
     \begin{table}
     \begin{center}\label{table:135}
    \begin{tabular}{|c|c|c|c|c|c|}
    \hline
      \#&$\overrightarrow{\Delta}=(\Delta r_1,\dots,\Delta r_5) \quad $   &  $\overrightarrow{r}=(r_0,r_1,\dots,r_{k-1})$&\multicolumn{2}{|c|}{\#rk matrices}\\
        \hline
         1&$(0,0,0,0,0)$\quad&$(5,5,5,5,5,5\dots,5)$&\textcolor{white}{aa}7\textcolor{white}{aa} &0\\
         2&$(1,0,0,0,0)$&$(5,4,4,4,4,4,\dots,4)$&5&0\\
         3&$(2,0,0,0,0)$&$(5,3,3,3,3,3,\dots,3)$&2&4\\
         4&$(1,1,0,0,0)$&$(5,4,3,3,3,3,\dots,3)$&6&0\\
         5&$(3,0,0,0,0)$&$(5,2,2,2,2,2,\dots,2)$&2&0\\
         6&$(2,1,0,0,0)$&$(5,3,2,2,2,2,\dots,2)$&2&2\\
         7&$(1,1,1,0,0)$&$(5,4,3,2,2,2,\dots,2)$&4&0\\
         8&$(3,1,0,0,0)$&$(5,2,1,1,1,1,\dots,1)$&1&0\\
         9&$(2,2,0,0,0)$&$(5,3,1,1,1,1,\dots,1)$&1&0\\
         10&$(2,1,1,0,0)$&$(5,3,2,1,1,1,\dots,1)$&1&1\\
         11&$(1,1,1,1,0)$&$(5,4,3,2,1,1,\dots,1)$&2&0\\
         12&$(3,2,0,0,0)$&$(5,2,0,0,0,0,\dots,0)$&1&0\\
         13&$(3,1,1,0,0)$&$(5,2,1,0,0,0,\dots,0)$&1&0\\
         14&$(2,2,1,0,0)$&$(5,3,1,0,0,0,\dots,0)$&1&0\\
         15&$(2,1,1,1,0)$&$(5,3,2,1,0,0,\dots,0)$&1&1\\
         16&$(1,1,1,1,1)$&$(5,4,3,2,1,0,\dots,0)$&2&0\\
         \hline 
    \end{tabular}
\end{center}
\caption{Case $s=5$: the sequences $\Delta$ and number of rank matrices. See Theorem \ref{5kthm} and Remark \ref{rankmatrixrem}.}\label{s5tab}
\end{table}
It is left to study the possible values of $r_{0,0},\dots r_{0,j-1}$ and $r_{1,1}, \dots, r_{1,j-2}$. Note that if $r_u$ is zero, then $r_{0,j-u}=r_{1,j-u-1}=0$ (the whole diagonal containing $r_u$ is zero.) \newline  
For case \#1, since the diagonals are Hilbert functions for AG algebras, then $r_{1,j-2}=\cdots=r_{1,3}=3$, and $r_{0,4}=\cdots=r_{0,j-1}=1$. It is left to count the possibilities for the matrix\par\vskip 0.2cm $\qquad\quad\begin{pmatrix}
    1&r_{0,3}&r_{0,2}&r_{0,1}&r_{0,0}\\
    3&3&r_{1,2}&r_{1,1}&r_{0,1}\\
    5&5&5&r_{1,2}&r_{0,2}\\
    5&5&5&3&r_{0,3}\\
    5&5&5&3&1    
\end{pmatrix}.$\par\vskip 0.2cm\noindent
By Lemma \ref{diaglem} (iv), $r_{1,2}=3$ and $r_{0,3}=r_{0,2}=1$. We have the following choices: 
\footnotesize $$\begin{pmatrix}
    1&0&0\\
    3&1&0\\
    5&3&1
\end{pmatrix},\begin{pmatrix}
    1&1&1\\
    3&1&1\\
    5&3&1
\end{pmatrix},\begin{pmatrix}
    1&0&0\\
    3&2&0\\
    5&3&1
\end{pmatrix},\begin{pmatrix}
    1&1&0\\
    3&2&1\\
    5&3&1
\end{pmatrix},\begin{pmatrix}
    1&1&1\\
    3&2&1\\
    5&3&1
\end{pmatrix},\begin{pmatrix}
    1&1&0\\
    3&3&1\\
    5&3&1
\end{pmatrix},\begin{pmatrix}
    1&1&1\\
    3&3&1\\
    5&3&1
\end{pmatrix}$$
\normalsize
For cases  \#2,4,7,11, and 16, (i.e. $r_1=4$, and therefore $r_{1,j-2}=3$) the second diagonal is $(1,3,4^{k-1},3,1)$. Hence, for $k\ge 4$, the possible rank matrices follow from the case $(1,3,4^k,3,1)$, (see Table \ref{s4tab}).
\newline 
For cases \#3,6,10, and 15 ($r_1=3$ and $r_2\ge 2$), $r_{1,j-2}$ is either 2 or 3. If it is equal to 3 (respectively 2), then the possible rank matrices follow from the case $(1,3^{k+1},1), k\geq 4$ (Table \ref{s=3tab}), (respectively $(1,2,3^{k-3},2,1), k\geq 4$, (Table \ref{table:123})).\newline 
Where we have $r_1=3$ and $r_2=1$, cases \#9 and \#14, we must have $r_{1,j-2}=3$. Since otherwise if $r_{1,j-2}=2$ the difference between the second and third diagonal of the rank matrix is given by $(1,2,3,3,\dots )-(1,1,1,\dots)=(1,1,2,2,\dots )$; that is not an O-sequence contradicting the condition (ii) in Lemma \ref{ranklem}. Thus the number of rank matrices in these two cases are given by the number of corresponding rank matrices for $(1,3^{k+1},1)$; i.e. cases \#3 and \#5 in Table~\ref{s=3tab}.\newline
For the remaining cases, \#5,8,12, and \#13 where $r_1=2$, we have $r_{1,j-2}=2$, the possible rank matrices follow from the case $(1,2^{k-1},1)$ (Table \ref{table:122}).\end{proof}
\begin{remark}\label{rankmatrixrem}
    The number of rank matrices for each row in Tables \ref{s4tab} and \ref{s5tab} is divided into two columns where the left (respectively, right) column indicates the number of rank matrices for which the diagonal after the main diagonal is the Hilbert function of some codimension three (respectively, codimension two) AG algebra. 
\end{remark}

Now we count the number of rank matrices when $s=5$ and $k=3$.
\begin{theorem}\label{5kbthm}
    Let $T=(1,3,5^k,3,1)$ and $k=3$ then there are 43 potential Jordan degree types for pairs $(A,\ell)$.
\end{theorem}
\begin{proof}
When $k=3$ the vector $\overrightarrow{r}$ has 3 non-zero entries. The number of rank matrices for $\overrightarrow{r}=(5,2,1)$ is equal to the sum of the number of rank matrices in rows \#8 and \#13 in Table \ref{s5tab}. More precisely, the rank matrices for $\overrightarrow{r}=(5,2,1)$ such that $(r_{0,3},r_{1,2},r_{1,2},r_{0,3})$ is the one or zero vector corresponding to \#8 and \#13 respectively.\par
\noindent Similarly, the number of rank matrices for $\overrightarrow{r}=(5,3,1)$ is equal to the number of rank matrices in the rows \#9 and \#14 where we have $(r_{0,3},r_{1,2},r_{1,2},r_{0,3})$ equal to one and zero vector respectively.\par
\noindent The number of rank matrices for $\overrightarrow{r}=(5,3,2)$ is given by the sum of rank matrices counted in rows \#6, \#10, and \#15. The number of rank matrices in \#6 corresponds to the case where $(r_{0,3},r_{1,2},r_{1,2},r_{0,3})=(1,2,2,1)$ and $(r_{0,2},r_{1,1},r_{0,2})=(1,2,1)$. Row \#10 counts the number of rank matrices where $(r_{0,3},r_{1,2},r_{1,2},r_{0,3})=(1,1,1,1)$ and $(r_{0,2},r_{1,1},r_{0,2})=(1,1,1)$ and row \#15 counts the rank matrices where $(r_{0,3},r_{1,2},r_{1,2},r_{0,3})=(1,1,1,1)$ and $(r_{0,2},r_{1,1},r_{0,2})=(0,0,0)$.\par
\noindent Finally, rows \#4,\#7,\#11, and \#16 correspond to $\overrightarrow{r}=(5,4,3)$. We notice that there are only 2 rank matrices where $(r_{0,3},r_{1,2},r_{1,2},r_{0,3})=(1,3,3,1)$ and $(r_{0,2},r_{1,1},r_{0,2})=(1,3,1)$ which is either $r_{0,0}=1$ or $r_{0,0}=0$. Therefore, the number of rank matrices where  $\overrightarrow{r}=(5,4,3)$ is 4 less than the sum of the rank matrices  equal to the sum of the number of rank matrices in \#4,\#7,\#11, and \#16 that is equal to 10.  We conclude that when $k=3$ there are 43 potential rank matrices.
\end{proof}

\subsubsection{Case $s=6$.}
Let $T=(1,3,6^k,3,1)$. The partition $\Delta=(3,3,0,0,\dots)$ that corresponds to $r=(6,3,0,0,\dots)$ cannot occur in this case since otherwise, if $T(1)=(1,2,3^{k-1},2,1)$ the difference between the first two diagonals would not be an $O$-sequence, and if $T(1)=(1,3^{k+1},1)$, then the submatrix $$\begin{pmatrix}
	r_{0,j-1}&r_{0,j-2}\\
	s&r_1
\end{pmatrix}=\begin{pmatrix}
1&0\\
3&3
\end{pmatrix}$$ would contradict Lemma \ref{ranklem}.  The number of possible sequences for $\overrightarrow{\Delta}$ is therefore  \small $$p(0)+p(1)+p(2)+p(3)+p(4,3)+p(5,3)+(p(6,3)-1)=1+1+2+3+4+5+(7-1)=22.$$
Similar to previous cases we count potential JDT or rank matrices for $s=6$ and $k\geq 5$ recursively using previous results that count the rank matrices for $s=3,4,5$ and $k\geq 4$.
\normalsize
\begin{theorem}\label{s=6thm}
    Let $T=(1,3,6^k,3,1)$, $k\ge 5$. There are 65 potential Jordan degree types for $(A,\ell)$ where  $A\in Gor(T)$ and $\ell\in A_1$. 
\end{theorem}
\begin{proof}

\par\noindent
   The possible sequences for $\overrightarrow{\Delta}$ are given in Table \ref{s6tab}. We first note that we count the rank matrices for AG algebras $A$ and linear forms $\ell\neq 0 $ so we have $r_1\neq 0$ and $\Delta r_0 <r_0=6 $. Therefore,  $r_{0,j-1}=1$. We claim that the second diagonal, $(1,r_{1,j-2}, r_1,\dots , r_1, r_{1,j-2},1 )$ is the Hilbert function of some codimension three AG algebra. Since otherwise the difference between the main and second diagonal is $(1,2,h,\dots )$ where $h = r_0 - r_{1,j-2}=4,5 $ meaning that this  difference is not an O-sequence contradicting Lemma \ref{ranklem} (ii).

It is left to study the possible values of $r_{0,0},\dots ,r_{0,j-1}$ and $r_{1,1}, \dots, r_{1,j-2}$. Note that if $r_u$ is zero, then $r_{0,j-u}=r_{1,j-u-1}=0$ (the whole diagonal containing $r_u$ is zero). \newline  
For case \#1, since the diagonals are Hilbert functions for AG algebras, then $r_{1,j-2}=\cdots=r_{1,3}=3$, and $r_{0,4}=\cdots=r_{0,j-1}=1$. By Lemma \ref{diaglem} (iv), $r_{1,2}=3$ and $r_{0,3}=r_{0,2}=1$. It is left to count the possibilities of the matrix\par\vskip 0.2cm $\begin{pmatrix}
    1&1&r_{0,2}&r_{0,1}&r_{0,0}\\
    3&3&3&r_{1,1}&r_{0,1}\\
    6&6&6&3&r_{0,2}\\
    6&6&6&3&1\\
    6&6&6&3&1    
\end{pmatrix}$\par\vskip 0.2cm\noindent
and we have the following choices for the upper right corner:  $$\begin{pmatrix}
    1&0&0\\
    3&1&0\\
    6&3&1
\end{pmatrix},\begin{pmatrix}
    1&1&1\\
    3&1&1\\
    6&3&1
\end{pmatrix},\begin{pmatrix}
    1&0&0\\
    3&2&0\\
    6&3&1
\end{pmatrix},$$ $$\begin{pmatrix}
    1&1&0\\
    3&2&1\\
    6&3&1
\end{pmatrix},\begin{pmatrix}
    1&1&1\\
    3&2&1\\
    6&3&1
\end{pmatrix},\begin{pmatrix}
    1&1&0\\
    3&3&1\\
    6&3&1
\end{pmatrix},\begin{pmatrix}
    1&1&1\\
    3&3&1\\
    6&3&1
\end{pmatrix}$$
For cases  \#2,4,7,12,19 and \#29, (i.e. $r_1=5$, $r_2\ge 4$ the second diagonal is $(1,3,5^{k-1},3,1)$. Hence, for $k\ge 5$, the possible rank matrices follow from the case $(1,3,5^k,3,1)$.\newline 
If $r_1=4$, the second diagonal is $(1,3,4^{k-1},3,1)$. Therefore, the possible rank matrices follow from the table of $(1,3,4^k,3,1)$. These are cases \#3,6,10,11, 17,18, 26,27 and 28 (see Table \ref{s4tab})

For cases \#5,9,15,16,24 and 25, $r_1=3$, then the possible rank matrices follow from the case $(1,3^{k+1},1)$ (Table \ref{s=3tab}).\newline 
  \begin{table}
     \begin{center} 
    \begin{tabular}{|c|c|c|c|}
    \hline
      \#&$\overrightarrow{\Delta}=(\Delta r_1,\dots,\Delta r_6)  $   &  $\overrightarrow{r}=(r_0,r_1,\dots,r_{k-1})$&\# rk matrices\\
        \hline
         1&$(0,0,0,0,0,0)$&$(6,6,6,6,6,6,6,\dots,6)$&7\\
         2&$(1,0,0,0,0,0)$&$(6,5,5,5,5,5,5,\dots,5)$&7\\
         3&$(2,0,0,0,0,0)$&$(6,4,4,4,4,4,4,\dots,4)$&5\\
         4&$(1,1,0,0,0,0)$&$(6,5,4,4,4,4,4,\dots,4)$&5\\
         5&$(3,0,0,0,0,0)$&$(6,3,3,3,3,3,3,\dots,3)$&2\\
         6&$(2,1,0,0,0,0)$&$(6,4,3,3,3,3,3,\dots,3)$&6\\
         7&$(1,1,1,0,0,0)$&$(6,5,4,3,3,3,3,\dots,3)$&6\\
         8&$(3,1,0,0,0,0)$&$(6,3,2,2,2,2,2,\dots,2)$&2\\
         9&$(2,2,0,0,0,0)$&$(6,4,2,2,2,2,2,\dots,2)$&2\\
         10&$(2,1,1,0,0,0)$&$(6,4,3,2,2,2,2,\dots,2)$&4\\
         11&$(1,1,1,1,0,0)$&$(6,5,4,3,2,2,2,\dots,2)$&4\\
         12&$(3,2,0,0,0,0)$&$(6,3,1,1,1,1,1,\dots,1)$&1\\
         13&$(3,1,1,0,0,0)$&$(6,3,2,1,1,1,1,\dots,1)$&1\\
         14&$(2,2,1,0,0,0)$&$(6,4,2,1,1,1,1,\dots,1)$&1\\
         15&$(2,1,1,1,0,0)$&$(6,4,3,2,1,1,1,\dots,1)$&2\\
         16&$(1,1,1,1,1,0)$&$(6,5,4,3,2,1,1,\dots,1)$&2\\
         17&$(3,2,1,0,0,0)$&$(6,3,1,0,0,0,0,\dots,0)$&1\\
         18&$(3,1,1,1,0,0)$&$(6,3,2,1,0,0,0,\dots,0)$&1\\
         19&$(2,2,2,0,0,0)$&$(6,4,2,0,0,0,0,\dots,0)$&1\\
         20&$(2,2,1,1,0,0)$&$(6,4,2,1,0,0,0,\dots,0)$&1\\
         21&$(2,1,1,1,1,0)$&$(6,4,3,2,1,0,0,\dots,0)$&2\\
         22&$(1,1,1,1,1,1)$&$(6,5,4,3,2,1,0,\dots,0)$&2\\
         \hline 
    \end{tabular}
\end{center}
\caption{Case $s=6$: the sequences $\Delta$ and number of rank matrices.}\label{s6tab}

\end{table}\vskip 0.3cm
\end{proof}

\begin{theorem}\label{6bthm}
    Let $T=(1,3,6^k,3,1)$. For $k=3$ there are 52 and for $k=4$ there are 61 potential Jordan degree types for pairs $(A,\ell)$.
\end{theorem}
\begin{proof}
Let $k=3$ and note that $\overrightarrow{r}$ has three non-zero entries. Similar to the proof of Theorem \ref{5kbthm} we consider the rows of Table \ref{s6tab} having the same $(r_0,r_1,r_2)$.\par Suppose $\overrightarrow{r}=(6,3,1)$ that is given by first three entries of $\overrightarrow{r}$ in \#12 and \#17 in Table \ref{s6tab}. So the number of rank matrices in this case is equal to the sum of the rank matrices given in \#12 and \#17 that is equal to two. \newline 
The first three entries of \#8, 13 and 18 are equal to  $\overrightarrow{r}=(6,3,2)$. So the number of rank matrices in this case is equal to the sum of rank matrices of the corresponding rows that is equal to 4.\newline
Similarly, for $\overrightarrow{r}=(6,4,2)$ there are 5 rank matrices that is the sum of the rank matrices in \#9, 14, 19, and 20. \newline
The case $\overrightarrow{r}=(6,4,3)$ corresponds to \#6, 10, 15, and 21. Notice that the rank matrices for $k=3$ and row \#6 correspond to the cases where we have $(r_{0,3},r_{1,2},r_{1,2},r_{0,3})=(1,3,3,1)$ and $(r_{0,2},r_{1,1},r_{0,2})=(1,3,1)$ and there are only two possible rank matrices in this case. Thus the total number of rank matrices for $\overrightarrow{r}=(6,4,3)$ is equal to 10. \newline
Finally, $\overrightarrow{r}=(6,5,4)$ is given in rows \#4, 7, 11, 16, and 22. Rank matrices corresponding to \#4 cannot occur for $k=3$, since they correspond to having $(r_{0,3},r_{1,2},r_{1,2},r_{0,3})=(1,4,4,1)$ that is not possible. And similar to the previous case there are only 2 rank matrices corresponding to \#6 for $k=3$. Thus the numbe of possible rank matrices for $\overrightarrow{r}=(6,4,3)$ is equal to 10. \newline
We conclude that for $k=3$ there are 52 potential rank matrices or equivalently Jordan degree types for pairs $(A,\ell)$.\par 
Now we show there are 61 potential rank matrices when $k=4$ using previous theorems for $s=3,4,5$ and $k=3$. We consider the first four entries of $\overrightarrow r$ in Table \ref{s6tab}. Thus rows \#14 and 20 correspond to $(r_0,r_1,r_2,r_3)=(6,4,2,1)$ and using the count for $s=4$ and $k=3$ we get that the number of rank matrices in this case is equal to the sum of rank matrices in in rows \#6 and 10 in Table \ref{s4tab} that is equal to $1+1=2$. Rows \#10, 2 and 21 correspond to $(6,4,3,2)$ and the count it equal to the sum of rank matrices of rows \#4, 7, and 11 in Table \ref{s4tab} which correspond to $(4,3,2)$ and is equal to $4+2+2=8$. Finally, the first four entries of $\overrightarrow r$ of rows \#7, 11, 16, and 22 is $(6,5,4,3)$. The number of rank matrices in this case is equal to the rank matrices for $s=5, k=3$ and having the $\overrightarrow r=(5,4,3)$. Thus using the last part of the proof of Theorem \ref{5kbthm} we get that in this case there are $10$ possible rank matrices (that is equal to the sum of number of rank matrices in \#7, 11, 16, and 22 minus $4$). We conclude that the number of rank matrices for $s=6$ and $k=4$ is $4$ less than the number of rank matrices for $s=6$ and $k\geq 5$ that is equal to $61$.
\end{proof}

\section{Actual JDT for algebras $A\in \Gor(T)$.}
We first classify three kinds of parts for JDT occurring for algebras in $\Gor(T)$, $T=(1,3,T_1,s^k,T_1^r,3,1)$ as $k$ increases (Section \ref{threesec}).  We then state the second part of our main results, that the JDT occurring for almost constant $T$ with $s\le 5$ are completely determined, and are exhibited in the appended Tables (Section \ref{mainsec2}). We then illustrate the methods we used to verify the tables (Section \ref{exsec}), and set out some questions and problems (Section \ref{questsec}).
\subsection{Three kinds of growth of JDT parts.}\label{threesec}
Recall that the JDT for any $A\in \Gor(T)$ for any $T$ is symmetric (Lemma~\ref{JDTsymlem}).
As the multiplicity $k$ of the Sperner number increases for the almost constant $T$ of Equation~\eqref{Taceq}, there are three kinds of growth of the parts that may occur for the  potential JDT for $A\in \Gor(T)$.  Here, we are considering $k\ge 3$ (or $k\ge 1$ if $s=3$), and related Gorenstein algebras, where $A(j)=R/\Ann F(j)$ and $F(j-1)=\ell\circ F(j)$ for a fixed linear form $\ell\in R_1$. By Lemma \ref{maphilblem} there is a length-$s$ punctual scheme $\Z\subset \mathbb P^2$ defined by the ideal $I_\Z$, so $F(j)\in (I_Z)_j^\perp$, and we choose $F(j)$ general enough so $H(R/\Ann F(j))=\Sym(j, H(R/I_\Z))$ (Equation \eqref{SymHeq}). Here $\tau (\Z)=$ the lowest degree $i$ such that $H(R/I_\Z)_i=s$. For $T$ almost constant, and $s\in\{4,5,6\}$ we assume $j\ge 6$.
\begin{definition}\label{partsdef} Let $T=(1,3,T_1,s^k,T_1^r,3,1)$ be a codimension three Gorenstein sequence (Equation \eqref{Gorseqeqn}), and let $j$ be the socle degree of $T$. We fix a punctual scheme $\Z\subset \mathbb P^2$ of Hilbert function
$H=(1,3, T_1,\overline{s})$, and consider $A(j)=R/\Ann\, F(j)$ with $F(j)$ general enough in $(I_\Z)_j^\perp$.  Recall that $H(A(j))=\Sym(H,j)$. We define three kinds of parts of the Jordan degree type $\mathcal S(R/\Ann F(j))$, as $j$ varies for such a related set $A(j)=R/\Ann F(j)$, where $F(j-i)=\ell^i\circ F(j)$ for $j-i\ge 2\tau(\Z)+2$.
\begin{enumerate}
\item  {\it lengthening} parts of the form $(j-a)_i$ where $i$ and $a$ are fixed as $j$ increases;
\item  {\it repeated} parts, of the form $(w\uparrow_i^{j+1-i-w})$, where $w$ and $ i$ are fixed as $j$ increases, and $w\le s$;
\item {\it sporadic} parts of the form  $(w_{i, j+1-i})$ where $w$ and $i$ are fixed as $j$ increases.
\end{enumerate}
\end{definition}
\begin{example} When $T=(1,3^k,1)$ the Table \ref{s=3tab} entry \#4 is $\overrightarrow{\Delta}=(1,1,0)$ and $\mathcal S=((j+1)_0, 2\uparrow_1^{j-2}, 1_{1,j-1})$ which has the lengthening part $(j+1)_0$, the repeated parts $(2\uparrow_1^{j-2})$ and the sporadic parts $(1_{1,j-1})$.
\end{example} 
Given $T$ from Equation \eqref{Taceq} and a rank matrix $M_{A,\ell}$ for $A\in \Gor(T)$, recall from Equation \ref{deltaeq} the sequence $\overrightarrow{\Delta}=(\Delta r_1,\ldots,\Delta r_{k-1})$, the first differences $\Delta r_i=r_{i-1}-r_i$ of the sequence $\overrightarrow{r}=(r_0,r_1,\ldots,r_{k-1})$ where $r_0=s$. We consider a set of related JDT $\mathcal S(j)=\mathcal S_{F(j),x}$ where $F(j)\in S_j$ and $F(j-i)=x^i\circ F(j)$, of Hilbert function $H(F(u))=(1,3,s^{u-3},3,1)$. We denote by $W_{\mathcal S(j)}$ the multi-set (with multiplicities)
\begin{equation}\label{Weqn} W_{\mathcal S(j)}=\{w \mid \exists (w\uparrow_i^{j+1-i-w}) \text { in the JDT $\mathcal S(j)$}\}.
\end{equation}
Recall that the conjugate partition $P^\vee$ to $P$ is obtained by switching rows and columns in the Ferrers diagram of $P$. We denote the sum of entries of $\overrightarrow{\Delta}$ by $|\Delta|$.
\begin{proposition}\label{partsprop} Let $s\le 6$ be fixed, let $T$ satisfy Equation \eqref{Taceq}, and fix a sequence $\mathcal S(j)$ of related JDT for $F(j)$ as for Equation \eqref{Weqn} above. Consider
$\overrightarrow{\Delta}=\overrightarrow{\Delta}(j)$ formed by Equation \eqref{deltaeq} from an AG algebra $A(j)=R/\Ann(F(j))$ for some $j\ge 6$. Assume that $k\ge s+1$.
\begin{enumerate}[(a)]
\item We have $W_{\mathcal S(j)}=(\overrightarrow{\Delta})^\vee$, the conjugate partition of $\overrightarrow{\Delta}$, independent of $j$.       
\item The total number of lengthening parts in the set of JDT $\mathcal S(j)$ is $(s-|\Delta|)$.  Also, a lengthening part is at least $j-3$. 
\end{enumerate}
\end{proposition}
\begin{proof} By Corollary \ref{deltacor} $\Delta_i=0$ for $i\ge s+1$, so we assume $i\ge s+1$. We prove (a) by induction on $(s-r_1)$, the largest part of $\overrightarrow{\Delta}$.  When this is $1$, so $\overrightarrow{\Delta}=(1,1,\ldots,1_p,0,\ldots)$, then the $k\times k$ submatrix at the center of $M_{A,\ell}$ in Equation \eqref{Meq} satisfies
\small
\begin{equation}\label{aux1eq}
\begin{pmatrix}
s&s-1&s-2&s-3&\ldots\color{red}& s-p+1&s-p&s-p&\ldots&\\
&s&s-1&s-2&&\ldots&\color{red}s-p+1&s-p&s-p\ldots \\
&&s&s-1&&\ldots &&\color{red}s-p+1&s-p\ldots&\\
&&&\ddots&&&\ddots &&&\\
&&& &s&\ldots &&&\color{red} s-p+1\\
\end{pmatrix}
\end{equation}
\normalsize
The JDT matrix JDT$_{A,\ell}$ of Definition \ref{def:jordan matrix} will have in this $k\times k$ region $$1=(s-p+1)+(s-p)-2(s-p)$$  on the locations shown red in the above matrix of Equation \eqref{aux1eq} (after the first row). These are on the $p-1$ labelled diagonal of $M_{A,\ell}$ so each corresponds to a part $p_i$, that is, $p$ in degree $i$ where  $i$ is the row index in $M_{A,\ell}$: thus, $\mathcal S_{A,\ell}$ has corresponding parts $(p\uparrow_3^{j-p-3})$.\par
Now fix $t\ge 2$ and assume (a) is true for first entries of $\overrightarrow{\Delta}$ less than $t$, and consider a particular $\overrightarrow{\Delta}(0)$ with first entry $t$,  that is nonzero in degrees $0,\ldots,p-1$, then zero. We can see that the last entry of each row of the $k\times k$ submatrix of 
JDT$_{A,\ell}$ is one: that is, on its $p-1$ labelled diagonal (analogous to that shown red in the $k\times k$ matrix $M_{A,\ell}$ of Equation \eqref{aux1eq}), so this corresponds to a repeating part $(p\uparrow_i^{j+1-p-i})$ in $\mathcal S$. Removing that layer, using the induction hypothesis, we see that the repeating parts of $\mathcal S_{A,\ell}$ correspond to the conjugate partition $\Delta^\vee$ to $\Delta$.  \par
We now show (b).  The remaining entries of the JDT matrix JDT$_{A,\ell}$ outside of the repeating parts, are symmetrically located. Since the maximum height of $T$ due to repeated parts is just the length $|\Delta|$, by the construction of the matrix $M$ and the related JDT matrix, the only way to attain the height $s$ is for there to be $(s-|\Delta|)$ lengthening parts. 
\par
The last statement of (b), that a lengthening part is at least $j-3$, follows from the next result, Lemma \ref{lengthenlem}.
\end{proof}

\begin{example}\label{reppartex}
(a). Entry \#34 in Table \ref{newk=5table} is  $\mathcal S_{A,x}=((j_{0,1},2\uparrow_2^{j-3},1\uparrow_1^{j-1},1_{2,j-2}) $ for $A=R/I, I=\Ann(X^{j-1}Y+XY^{j-1}+Y^{j-2}Z^2).$ 
 Here $s=5$, the repeated parts are $2$ and $1$, and the two lengthening parts are $j_{0,1}$, while 
$\Delta=(2,1,0)$ whose dual is (also) $(2,1)$. \par
{b}. Entry \#43 in Table \ref{newk=5btable} has $\mathcal S_{A,x}=((4\uparrow_0^{j-3},2_{1,j-2},1\uparrow_2^{j-2})$ for $A=R/\Ann (X^3Y^{j-3}+XY^{j-2}Z+Z^j)$. Here the repeated parts are $4$ and $1$, while $\Delta(A)=(2,1,1,1)$, and there are no lengthening parts.\par
(c). In codimension two, the repeating parts of the associated JDT are also obtained as the conjugate partition to $\overrightarrow{\Delta}=(\Delta r_1, \Delta r_2,\Delta r_3)$. For example $\overrightarrow{\Delta}=(1,1,0)$ has repeating part $2$, and $\overrightarrow{\Delta}=(2,0,0)$ has repeating parts  $(1,1)$. See Example \ref{cod2ex}, and Tables \ref{table:122} and \ref{table:123} above.
\end{example}

\begin{lemma}\label{lengthenlem} Let $T$ be an almost constant Gorenstein sequence, with Sperner number $s=3,4,5,$ or $6$.  Then the possible lengthening parts for $\mathcal S(A,\ell), A\in \Gor(T)$ satisfy the following restrictions,
\begin{enumerate}[-]
	\item in degree 0: $(j-a)_0$ with $a=-1,0,1$,
	\item in degree 1: $(j-a)_1$ with $a=0,1,2$, 
	\item in degree 2: $(j-a)_2$ with $a=1,2,3$. 
\end{enumerate}  
No other lengthening parts can occur.
\end{lemma}
\begin{proof}
Since the sequence $\Delta$ is non-increasing, non-negative and its elements add up to at most $s$, then $\Delta_i=0$ for $i\ge s$, and therefore $r_s=r_{s+1}=\dots=r_{k-1}$. Furthermore, if $\begin{pmatrix}
	\alpha&\beta\\
	\gamma&\delta
\end{pmatrix}$ is an adjacent submatrix of the rank matrix such that $\gamma=\delta$, then $\alpha=\beta$.  
The rank matrix can then be written as follows:

$$\begin{pmatrix}
	1&r_{0,j-1}&\dots&r_{0,j-s}&r_{0,j-s-1}&\color{blue}r_{0,j-s-2}&\color{blue}r_{0,j-s-2}&\color{blue}\dots&\color{blue}r_{0,j-s-2}&r_{0,1}&r_{0,0}\\
	&3&r_{1,j-2}&\dots&r_{1,j-s-1}&\color{blue}r_{1,j-s-2}&\color{blue}r_{1,j-s-2}&\color{blue}\dots&\color{blue}r_{1,j-s-2}&r_{1,1}&r_{0,1}\\
	&&s&r_1&\dots&\color{blue}r_s&\color{blue}r_s&\color{blue}\dots&\color{blue}r_s&r_{1,j-s-2}&r_{0,j-s-2}\\
	&&&s&r_1&\dots&\color{blue}r_s&\color{blue}\ddots&\color{blue}\vdots&\vdots&\vdots\\
	&&&&\ddots&\ddots&&\color{blue}\ddots&\color{blue}r_s&r_{1,j-s-2}&r_{0,j-s-2}\\
	&&&&&\ddots&\ddots&&\color{blue}r_s&r_{1,j-s-2}&r_{0,j-s-2}\\
	&&&&&&\ddots&\ddots&\vdots&r_{1,j-s-1}&r_{0,j-s-1}\\
	&&&&&&&s&r_1&\vdots&r_{0,j-s}\\
	&&&&&&&&s&r_{1,j-2}&\vdots\\
	&&&&&&&&&3&r_{0,j-1}\\
	&&&&&&&&&&1
\end{pmatrix}$$
Furthermore, since $r_{0,a}=0$ if one of the elements on the same diagonal is 0 and 1 otherwise, then $r_{0,j-s}=r_{0,j-s-1}=r_{0,j-s-2}$

Define $J_i=r_i+r_{i+2}-2r_{i+1}$. The Jordan degree type matrix is therefore 
$$\begin{pmatrix}
	J_{0,0}&J_{0,1}&\dots&J_{0,s-1}&0&0&0&0&\dots&0&J_{0,j-2}&J_{0,j-1}&J_{0,j}\\
	&J_{1,1}&J_{1,2}&\dots&J_{1,s}&J_{1,s+1}&0&0&\dots&0&J_{1,j-2}&J_{1,j-1}&J_{0,j-1}\\
	&&J_{2,2}&\dots&\dots&J_{2,s+1}&0&0&\dots&0&J_{2,j-2}&J_{1,j-2}&J_{0,j-2}\\
	&&&J_0&\dots&\dots&J_{s-1}&0&\dots&0&0&0&0\\
	&&&&\ddots&&\ddots&\ddots&\ddots&&\vdots&\vdots&\vdots\\
	&&&&&\ddots&&\ddots&\ddots&0&0&0&0\\
	&&&&&&\ddots&&\ddots&J_{s-1}&0&0&0\\
	&&&&&&&\ddots&&\vdots&J_{2,s+1}&J_{1,s+1}&0\\
	&&&&&&&&\ddots&\vdots&\vdots&J_{1,s}&0\\
	&&&&&&&&&J_0&\vdots&\vdots&J_{0,s-1}\\
	&&&&&&&&&&J_{2,2}&J_{1,2}&\vdots\\
	&&&&&&&&&&&J_{1,1}&J_{0,1}\\
	&&&&&&&&&&&&J_{0,0}\\
\end{pmatrix}$$

Note that the lengthening parts come from the top right $3\times 3$-submatrix of the Jordan degree type matrix. This completes the proof of Lemma \ref{lengthenlem}.
\end{proof}
See Example \ref{4.10ex}, for an application of Lemma \ref{lengthenlem}.

\subsection{Classification of JDT for $s\le 5$.}\label{mainsec2}
We here report on the classification of all occurring JDT for graded Artinian algebras of almost constant Hilbert function $T$ of codimension three with Sperner number $3,4,5$ - that are given in the Tables Section \ref{tablesec}. The results of this section assume $\cha {\sf k}=0$. For $\cha {\sf k}$ finite see Remark \ref{inffieldrem}.\par
We first determine the occurring JDT for $T$ of Sperner number 3.
\begin{theorem}\label{k3prop} A. Let $T=(1,3^k,1)$. Suppose first that $k\ge 3$. Then \#1-8 in Table \ref{1351fig} is a complete table of the eight occurring JDT for algebras $A\in\Gor(T)$. \par 
B. When $k=2$, \# 2 is the same as \# 3 and \#6 is the same as \#7 but there is an extra \# 9. When $k=1$ only \# 1 and \# 8 are distinct. 
\end{theorem}
We next consider Sperner number 4.
\begin{theorem}\label{s=4thm}\footnote{In Table \ref{4ktable} we indicate the Jordan degree types that occur for Complete intersection algebras by CI and the ones that occur for connected sum algebras by cs, for more details on connected sum algebras see \cite{IMM}.}  Let $T=(1,3,4^k,3,1)$. When $k\ge 3$ there are exactly 26 JDT that actually occur, as presented in Table \ref{4ktable}. When $k=2$,  there are exactly 22 JDT, as then  \#6=\#7, \#10=\#11, \#16=\#18, and \#23=\#24 in Table \ref{4ktable}.  When $k=1$ there are exactly 12 JDT arising from Table \ref{4ktable} \#1,\#3=\#5,7,8,12,13,15,19,20,24,25, and \#26.
\end{theorem}
We next consider Sperner number 5.

\begin{theorem}\label{s=5thm} Let $T=(1,3,5^k,3,1)$.  When $k\ge 4$  (so $j\ge 7$) there are exactly 47 JDT, all distinct, that occur for $T$, listed inTables  \ref{newk=5table},\ref{newk=5btable}, where we give the JDT, a suitable dual generator (in all but one case) and the ideal, in terms of $j$. When $k=3$ there are identities \#9=\#11, \#10=\#12, \#29=\#30, and \#39=\#42, and there are exactly 43 distinct JDT. When $k=1$ there are $12$ distinct, occurring JDT - (Table \ref{13531table}). When $k=2$ there are $35$ distinct JDT: see Table  \ref{55table}, where we give in the last column the corresponding Table~\ref{newk=5table}/\ref{newk=5btable} entry).
\end{theorem}
\vskip 0.2cm\noindent
\begin{proof}[Proof of Theorem \ref{k3prop}, Theorem \ref{s=4thm} and Theorem \ref{s=5thm}.]
The restriction on the number of potential JDT are shown in each case of almost constant $T$ for $3\le s\le 5$ in Section \ref{potentialsec}, Theorem \ref{s=3thm} for $(s=3) $, Theorem \ref{4kthm} for $(s=4)$, Theorems \ref{5kthm} and \ref{5kbthm} for $s=5$ and 
Theorem \ref{s=6thm} for $s=6$.  The tables - and especially the dual generators we found were constructed using several techniques, and verified case by case. We give examples of the techniques we used to verify table entries in the next Section. We of course used {\sc Macaulay2} in verifying some dual generators for specific $j$, as a check, but our techniques apply - as they must - to arbitrary socle degree $j$.
\end{proof}
For Sperner number $6,\, T=(1,3,6^k,3,1)$ we did not construct tables of the actually occurring JDT so have not verified that each potential JDT from Theorem \ref{s=6thm} actually occurs.\par
\begin{remark}[When the characteristic $\cha{\sf k}$ is finite]\label{inffieldrem}  Many results in this area require a restriction to fields of characteristic zero, or of characteristic $p$ higher than the socle degree $j$ of $A$. For example, in characteristic $2$ the algebra $A=R/I, I=(x^2,y^2,z^2)=\Ann(F), F=XYZ, H(A)=(1,3,3,1)$ is not strong Lefschetz: any linear form $\ell$ satisfies $\ell^2=0$, and for a generic $\ell$, we have $\mathcal S_{A,\ell}=(2\uparrow_0^2,2_1)$, rather than the strong Lefschetz JDT $H^\vee=(4_0,(2_1)^2)$.   Example \ref{MAex} above with $F=X^3+Y^3+Z^3$ shows that $\ell=x+y+z$, although a strong Lefschetz element when the characteristic of $\sf k$ is not $3$, is not a strong Lefschetz element over characteristic $3$. Thus, ostensibly, the set $\mathcal J(T)$ of JDT for a given Gorenstein sequence $T$ might depend upon the characteristic of $\sf k$: that is, for certain fields of small characteristic, some of the entries in our Tables for $s=3,4,5$ might not appear.  However we do not have any such example. Our Macaulay dual generators are notable in that, with the exception of one case, their coefficients on monomials are one; and the linear form $\ell=x$ in all but the two cases $s=3$, Table \ref{1351fig}, entries \# 1,2, where $F=X^j+Y^j+Z^j$ and $\ell=x+y+z$ or $\ell=x+y$. Both these JDT occur for suitable Artinian algebras $A(j)=R/\Ann (F(j))$ and the linear form $x$.
\end{remark}
\subsection{Examples and Methods.}\label{exsec}
We here give examples illustrating the methods we used to determine the JDT in the tables. The following example illustrates how we obtain the JDT $\mathcal J(T)$ for a given $T$ and sequence $\overrightarrow{\Delta}$.
\begin{example}\label{4.10ex}
	Consider the Hilbert function $H=(1,3,4^k,3,1)$ and the sequence $(r_0,r_1,\dots,r_k)=(4,3,3,\dots,3)$, so $\overrightarrow{\Delta}=(1,0,\dots, 0)$. The second diagonal of the rank matrix can be either the Hilbert function of a codimension 2 algebra $(1,2,3^{k-1},2,1)$, or that of a codimension 3 algebra $(1,3^{k+1},1)$. These two cases are illustrated by the below examples. The matrices $M_{A,\ell}$ and $J_{A,\ell}$ below are an example of a rank matrix and its corresponding JDT matrix for the first case. The JDT is  $((j+1)_0,(j-1)_1,(j-3)_2,1\uparrow_1^{j-1})$. The matrices $M'_{A,\ell}$ and $J'_{A,\ell}$ are an example for the second case. The Jordan degree type in this case is  $((j+1)_0,((j-1)_1)^2,1\uparrow_2^{j-2})$. 
 \small
	\begin{equation*}
		M_{A,\ell}=\begin{pmatrix}
			1&1&1&1&1&\dots &1&1&1\\\
			&3&2&2&2&\dots &2 &2&1\\\
			&&4&3&3&\dots&3 &2&1\\
			&&&4&3&\dots & 3&2&1\\
			&&&&\ddots&\ddots&\vdots &\vdots&\vdots\\
			&&&&&\ddots&3&2&1\\
			&&& &&&4&2&1\\
			&&& &&&&2&1\\
			&&& &&&&&1\\
		\end{pmatrix},J_{A,\ell}=\begin{pmatrix}
		0&0&0&0&0&\dots &0&0&1\\\
		&1&0&0&0&\dots &0&1&0\\\
		&&1&0&0&\dots&1 &0&0\\
		&&&1&0&\dots & 0&0&0\\
		&&&&\ddots&\ddots&\vdots &\vdots&\vdots\\
		&&&&&\ddots&0&0&0\\
		&&& &&&1&0&0\\
		&&& &&&&1&0\\
		&&& &&&&&0\\
		\end{pmatrix}.
	\end{equation*}
	
\begin{equation*}
	M'_{A,\ell}=\begin{pmatrix}
		1&1&1&1&1&\dots &1&1&1\\\
		&3&3&3&3&\dots &3 &3&1\\\
		&&4&3&3&\dots&3 &3&1\\
		&&&4&3&\dots & 3&3&1\\
		&&&&\ddots&\ddots&\vdots &3&1\\
		&&&&&\ddots&3&3&1\\
		&&&& &&4&3&1\\
		&&&& &&&3&1\\
		&&&& &&&&1\\
	\end{pmatrix}, J'_{A,\ell}=\begin{pmatrix}
		0&0&0&0&0&\dots &0&0&1\\\
		&0&0&0&0&\dots &0&2&0\\\
		&&1&0&0&\dots&0 &0&0\\
		&&&1&0&\dots & 0&0&0\\
		&&&&\ddots&\ddots&\vdots &\vdots&\vdots\\
		&&&&&\ddots&0&0&0\\
		&&&& &&1&0&0\\
		&&&& &&&0&0\\
		&&&& &&&&0\\
	\end{pmatrix}.
\end{equation*}\normalsize
In both examples, the lengthening parts are given by the nonzero elements of the $3\times 3$ top right corner submatrix of the Jordan degree type matrix (see Lemma \ref{lengthenlem}). For instance, the entries 1 on the first row and 2 on the second row of the submatrix $\begin{pmatrix}
	0&0&1\\
	0&2&0\\
	0&0&0
\end{pmatrix}$ of $J'_{A,\ell}$ give $(j+1)_0$ and twice $(j-1)_1$. The repeated entries $1$ on the diagonal of each JDT give the repeated portion  $1\uparrow_2^{j-2}$ for the first example, and $1\uparrow_1^{j-1}$ for the second.  \par

\end{example}
For all of the table entries for $s=3$ and $s=4$, and all but one for $s=5$ we are able to provide a dual generator for an algebra having the required JDT. An exception is \# 3 for $s=5$, where the dual generator can be found for small cases
(we calculated up to $j=20$ using {\sc Macaulay2}, but were unable to find a pattern). In this case we instead simply show (next) that for arbitrary socle degree $j$, the algebra $A=R/I(j)$ is Gorenstein of the expected Hilbert function.
\begin{example}
For each of the ideals in the third column of the table \ref{newk=5table} we can compute a Gr\"{ o}bner basis (standard basis)  using homogeneous degree rvlex order with $z > y> x$. \\
We consider the case \#3 for $T=(1,3,5^k,3,1)$ of socle degree $j=k+3$, where the ideal $I(j)$ (we will call $I$ for simplicity) is
\begin{equation}
I=\left(xy-yz,x^3-z^3,y^3-x^2z-y^2z-xz^2+yz^2-z^3,x^{j-7}yz^5,x^{j-5}z^4\right).
\end{equation}
It is straightforward to verify that the Hilbert function $H(R/I)=T$.
Since $yz\equiv xy$ and $z^3\equiv x^3 \ \text{mod}\,I$, one can check that 
$$I=\left(z^3-x^3, zy-yx, y^3-z^2x-y^2x-zx^2+yx^2-x^3,zx^{j-2},yx^{j-2}\right)$$ and a standard basis for $I$ is given by 
$$\left(z^3-x^3, zy-yx, y^3-z^2x-y^2x-zx^2+yx^2-x^3,zx^{j-2},yx^{j-2},x^{j+1}\right).$$
The strings describing the JDT of the algebra $R/I$ are the following: 
\begin{itemize}
\item a length $(j+1)$ string starting at degree $0$: $(1, x, \dots, x^i, \dots, x^j)$
\item two length $(j-2)$ strings starting both at degree one: 
$$(z, zx, \dots, zx^i, \dots, zx^{j-3}) \ \text{and} \ (y, yx, \dots, yx^i, \dots, yx^{j-3})$$
\item two length $(j-2)$ strings, both starting at degree two: 
$$(z^2, z^2x, \dots, z^2x^i, \dots, z^2x^{j-3}) \ \text{and} \  (y^2, y^2x, \dots, y^2x^i, \dots, y^2x^{j-3}).$$
\end{itemize}
These strings give the JDT $(\left(j+1\right)_0, \left(( j-2\right)_{1,2})^2)$ for the algebra $A=R/I$. 
\par It remains to show that $A$ is Gorenstein. We check that the socle of the graded algebra $A$ has dimension one. 
\par
From the five strings that describe the JDT of $A$, one can see that for $2\leq i\leq j-2$, a basis of $A_i$ is given by $\left(z^2x^{i-2}, y^2x^{i-2},zx^{i-1},yx^{i-1},x^i\right)$, 
$A_{j-1}=\langle z^2x^{j-3}, y^2x^{j-3}, x^{j-1}\rangle$,  $A_{j}=\langle x^{j}\rangle$.
\par
It is obvious that for $0\leq i \leq j-3$, the multiplication by $x$, $m_x: A_i \rightarrow A_{i+1}$ is injective. This means that any non zero homogeneous element $f$ of the socle of $A$ has degree at least $j-2$.
\par
We have $A_{j-2}=\langle z^2x^{j-4}, y^2x^{j-4},zx^{j-3},yx^{j-3},x^{j-2}\rangle$. \\
Let $(a_1, \dots , a_5)\in \sf k^5$, $f=a_1z^2x^{j-4}+a_2y^2x^{j-4}+a_3zx^{j-3}+a_4yx^{j-3}+a_5x^{j-2} \in A_{j-2}$ and suppose $xf=yf=zf=0$.
\begin{enumerate}[i.]
\item Since $zx^{j-2}\in I$ and $yx^{j-2}\in I$, $xf=0$ gives $a_1z^2x^{j-3}+a_2y^2x^{j-3}+a_5x^{j-2}=0$. Since $(z^2x^{j-3}, y^2x^{j-3}, x^{j-1})$ is a basis of $A_{j-1}$, 
we get $a_1=a_2=a_5=0$, so $f=a_3zx^{j-3}+a_4yx^{j-3}$.
\item Let $f=a_3zx^{j-3}+a_4yx^{j-3}$ and suppose $yf=a_3zyx^{j-3}+a_4y^2x^{j-3}=0$. We know that $zy-yx\in I$, so $yf=a_3yx^{j-2}+a_4y^2x^{j-3}=a_4y^2x^{j-3}$. This means that $yf=0$ gives $a_4y^2x^{j-3}=0$, so $a_4=0$ and $f=a_3zx^{j-3}$.
\item For $f=a_3zx^{j-3}$ one can see that $zf=a_3z^2x^{j-3}=0$ if and only if $a_3=0$.
\end{enumerate} 
\par 
Now, looking at degree $j-1$, let $(a_1,a_2,a_3)\in \sf k^3$, $f=a_1z^2x^{j-3}+a_2y^2x^{j-3}+a_3x^{j-1}$ such that $xf=yf=zf=0$.\\
One can easily check that $xf=a_3x^j$, $yf=a_2x^j$ and $zf=a_1x^j$ so $xf=yf=zf=0$ if and only if $a_1=a_2=a_3=0$. 
\par
In conclusion, the socle of $A$ is one dimensional, generated by $x^j$.
\end{example}
\begin{example}
Consider the Hilbert function $T=(1,3,5^k,3,1)$ and \#5 in Table \ref{newk=5table}. We want to show that $F=X^j+X^{j-1}Y+X^{j-3}Z^3+Y^j$ and $A=R/\Ann (F)$ has associated JDT $((j+1)_0,(j-1)_1,(j-2)_{1,2},1\uparrow_2^{j-2})$;  Consider $F_x=x\circ F=X^{j-1}+X^{j-2}Y+X^{j-4}Z^3$, which corresponds to \#3 in Table \ref{4ktable} for $T_1=(1,3,4^{k-1},3,1)$ (or $F_1=X^j+X^{j-1}Y+X^{j-3}Z^3$ for $T=(1,3,4^k,3,1)$). If we have a proof that the dual generator $F_1$ has JDT $((j+1)_0,(j-1)_1,(j-2)_{1,2})$ and we know that $F$ gives the Hilbert function $H(R/\Ann(F))=T$, then we are done, since the Jordan degree type matrix for 
$B=R/\Ann(F_1)$ is
\small
\vskip 0.3cm
\begin{equation*}
		M_{B,x}=\begin{pmatrix}
			1&1&1&1&1&\dots &1&1&1\\\
			&3&3&3&3&\dots &3 &2&1\\\
			&&4&4&4&\dots&4 &3&1\\
			&&&4&4&\dots & 4&3&1\\
			&&&&\ddots&\ddots&\vdots &\vdots&\vdots\\
			&&&&&\ddots&4&3&1\\
			&&& &&&4&3&1\\
			&&& &&&&3&1\\
			&&& &&&&&1\\
		\end{pmatrix},J_{B,x}=\begin{pmatrix}
		0&0&0&0&0&\dots &0&0&1\\\
		&0&0&0&0&\dots &1&1&0\\\
		&&0&0&0&\dots&0 &1&0\\
		&&&0&0&\dots & 0&0&0\\
		&&&&\ddots&\ddots&\vdots &\vdots&\vdots\\
		&&&&&\ddots&0&0&0\\
		&&& &&&0&0&0\\
		&&& &&&&0&0\\
		&&& &&&&&0\\
		\end{pmatrix}.
	\end{equation*}
\normalsize
These imply that the corresponding matrices for $A$ are (first find $M_{A,x}$ then $J_{A,x}$) are, adding respectively a first diagonal $(1,3,5^k,3,1) $ to $M_{A,x}$ and a first diagonal $(0,0,1,1,\dots,1,0,0)$ to $J_{A,x}$
\small
\vskip 0.3cm
\begin{equation*}
		M_{A,x}=\begin{pmatrix}
			1&1&1&1&1&\dots &1&1&1\\\
			&3&3&3&3&\dots &3 &2&1\\\
			&&5&4&4&\dots&4 &3&1\\
			&&&5&4&\dots & 4&3&1\\
			&&&&\ddots&\ddots&\vdots &\vdots&\vdots\\
			&&&&&\ddots&4&3&1\\
			&&& &&&5&3&1\\
			&&& &&&&3&1\\
			&&& &&&&&1\\
		\end{pmatrix},J_{A,x}=\begin{pmatrix}
		0&0&0&0&0&\dots &0&0&1\\\
		&0&0&0&0&\dots &1&1&0\\\
		&&1&0&0&\dots&0 &1&0\\
		&&&1&0&\dots & 0&0&0\\
		&&&&\ddots&\ddots&\vdots &\vdots&\vdots\\
		&&&&&\ddots&0&0&0\\
		&&& &&&1&0&0\\
		&&& &&&&0&0\\
		&&& &&&&&0\\
		\end{pmatrix}.
	\end{equation*}
\normalsize
The JDT matrix  $J_{A,x}$ implies that the JDT for $A$ is $((j+1)_0,(j-1)_1,(j-2)_{1,2},1\uparrow_2^{j-2})$.
\end{example} 
\subsection{Questions and problems.}\label{questsec}
First, can we extend these results to more complex Gorenstein sequences of codimension three, where the Sperner number $s$ is repeated $k$ times?

\begin{question}\label{4.13quest} Consider a general codimension three Gorenstein sequence $T$ as in Equation \ref{Gorseqeqn}: that is, $T=T(T_1,s,k)=(1,3,T_1, s^k, T_1^r, 3,1)$ where $(1,3,T_1,s)$ is an O-sequence, and $T_1^r$ is the reversal of the sequence $T_1$. As $k$ increases, with $T_1,s$ held constant, we expect\par

i.  a stable constant limit number of potential Jordan degree types for $T(T_1,s,k)$ as $T_1$ and $s$ are fixed but $k$ increases;\par
ii. there to be a list of potential Jordan degree types for these $T(T_1,s,k)$ as $T_1,s$ are fixed and $k$ increases, with structure analogous to those we found for almost constant $T$, that is, with the three kinds of parts of Definition \ref{partsdef} satisfying Proposition \ref{partsprop}. \par
iii. that the potential JDT found in (ii) all actually occur.

\end{question}\noindent
{\bf Remark}. We can do similar operations on the possible rank matrices as in the $T$ almost constant case: thus, we can bound the number of potential Jordan types, so answering (i).\par
We expect that parts (ii), (iii) of this question could be approached inductively, from simpler $T_1$ to more complex ones.  For example, to understand even $T=(1,3,5,6^k,5,3,1)$ we might need to have first determined the answers for simpler Hillbert functions as $T=(1,3,4,5^k,4,3,1)$. 

\begin{question}[Limits of Gorenstein algebras]
Let $T$ be a codimension three Gorenstein sequence.  There is an inclusion  $\Gor(T)\subset G_T$ of the family of graded Gorenstein algebras of Hilbert function $T$ in the family $G_T$ of all graded algebra quotients of $R={\sf k}[x,y,z]$ having Hilbert function $T$. \par
a. For which $T$ is the Zariski closure $\overline{\Gor(T)}=G_T$?  It is known that for $T=(1,3^k,1)$ the closure $\overline{\Gor(T)}=G_T$, but for $T=(1,3,5,3,1), (1,3,6,3,1)$ or $(1,3,6,6,3,1)$ the closure of $\Gor(T)$ is proper in $G_T$ \cite[\S 3.7]{AEIY}.\par
b. What are limits of JDT strata of $\Gor(T)$ in $G_T$?  Note that the JDT for algebras $A$ in  $ G_T$ need not be symmetric, even when $A $ is in the closure $\overline{\Gor(T)}$ \cite[Example 4.16]{AEIY}.
\end{question} 
\begin{question}  Do the actual tables - the set $\mathcal J(T)$ of JDT for $T$ - depend on the characteristic of $\sf k$?  See Remark \ref{inffieldrem}.
\end{question}
\begin{ack} The work was begun at the CIRM Conference ``Lefschetz Properties in Algebra, Geometry, and Combinatorics, II'' in 2019 at Luminy, France.  We appreciate contributions to our discussions by Alexandra Seceleanu, and Leila Khatami, who  were members of our codimension three working group at Lumini, and participated further. We thank Pedro Macias Marques and Daniele Taufer for comments and questions. We appreciate comments by the referee. The second author was supported by the grant VR2021-00472. Some part of the work was done when the second author was an NSF postdoc at the program in Commutative Algebra, during Spring 2024 semester at Simons Laufer Mathematical Sciences Institute (formerly MSRI). 

\end{ack}
   
\appendix
\newpage
\section{Tables of JDT in codimension 3 for almost constant $T$, when $s\le 5$.}\label{tablesec}
\footnotesize
\begin{table}[h]
\begin{center}
\footnotesize
$\begin{array}{|c|c|c||c|c|}
\hline
\#&\mathcal{S}&F&I=\Ann(F)&\ell\\
\hline\hline
1&((j+1)_0,((j-1)_1)^2)&X^j+Y^j+Z^j&xy,xz,yz,x^j-y^j,x^j-z^j&{\tiny{x+y+z}}\\
\hline
2&((j+1)_0,(j-1)_1,1\uparrow^{j-1}_1)&X^j+Y^j+Z^j&''&x+y\\
\hline
3&((j+1)_0,2\uparrow^{j-2}_1,1_1,1_{j-1})&X^j+XY^{j-1}+ZY^{j-1}&xz,xy-yz,z^2,x^j-xy^{j-1},x^j-zy^{j-1}&x\\
\hline
4&((j+1)_0,(1\uparrow^{j-1}_1)^2)& X^j+Y^j+Z^j&xy,xz,yz,x^j-y^j,x^j-z^j&x\\
\hline\hline
5&(j_0,j_1, (j-1)_1)& X^{j-1}Z+X^{j-2}Y^2&xz-y^2,yz,z^2,x^{j-1}y,x^j& x\\
\hline
6&(j_0,j_1,1\uparrow^{j-1}_1)& X^{j-1}Y+Z^j&xz,yz,y^2,x^j,x^{j-1}y-z^j& x\\
\hline\hline
7&(3\uparrow^{j-2}_0,1_1,1_{j-1})&X^2Y^{j-2}+Y^{j-1}Z&x^2-yz,xz,z^2,xy^{j-1},y^j&x\\
\hline\hline
8&(2\uparrow^{j-1}_0,1\uparrow^{j-1}_1)&XZ^{j-1}+Y^2Z^{j-2}&x^2,xy,y^2-xz,yz^{j-1},z^j&x\\\hline
\ast 9&(2\uparrow_0^2,2_1)\text{ if $k=2$ }&X(Z^2+Y^2)&x^2,yz,y^2-z^2,x(y^2-z^2)&x\\\hline
\end{array}$\par
\#9\,\text{only occurs for $k=2$}
\caption{Complete 8 JDT for $T=(1,3^k,1), k\ge 3, j=k+1$ (Theorem~\ref{k3prop}).}\label{1351fig}
\end{center}
\end{table}
\newpage
\begin{landscape}
\begin{table*}[h]
\begin{center}
$\begin{array}{|c|c|c|c|c|c|c|c|}
\hline
\#&\mathcal{S}&F&I=\Ann(F) \\
\hline\hline
1&((j+1)_0,((j-1)_1)^2,(j-3)_2)& X^j+X^{j-2}(Y^2+Z^2)\,\,\,CI&yz,y^2-z^2,x^{j-1}-x^{j-3}y^2\\
\hline
2&((j+1)_0,((j-1)_1)^2,1\uparrow_2^{j-2})& X^j+X^{j-1}Y+X^{j-2}Z^2+Z^j&y^2,yz,x^2y-xz^2,x^{j-2}z-z^{j-1},x^j-z^j\\
\hline
3&((j+1)_0,(j-1)_1,(j-2)\uparrow_1^2)& X^j+X^{j-1}Y+X^{j-3}Z^3 &yz,y^2,x^2y-z^3,x^{j-2}z,x^j-x^{j-3}z^3\\
\hline
4&((j+1)_0,(j-1)_1,(j-3)_2,1\uparrow_1^{j-1})& X^j+X^{j-2}Y^2+Z^j\,cs&yz,xz,y^3,x^{j-1}-x^{j-3}y^2,x^{j-2}y^2-z^j\\
\hline
5&((j+1)_0,(j-1)_1,2\uparrow_1^{j-2})& X^j+X^{j-1}Y+XZ^{j-1}\, cs&yz,y^2,x^2z,x^{j-2}y-z^{j-1},x^j-xz^{j-1} \\
\hline
6&((j+1)_0,(j-1)_1,2\uparrow_2^{j-3},1_{1,2,j-2,j-1})& X^j+X^{j-1}Y+XY^{j-1}+Y^{j-1}Z&xz,z^2,xy^2-y^2z,x^{j-1}-x^{j-2}y-y^{j-2}z,y^j\\
\hline
7&((j+1)_0,(j-1)_1,1_1,(1\uparrow_2^{j-2})^2,1_{j-1})& X^j+X^{j-1}Y+Y^j+Z^j\,cs&yz,xz,xy^2,x^{j-1}-x^{j-2}y-y^{j-1},y^j-z^j\\
\hline
8&((j+1)_0,(j-2)\uparrow_1^2,1\uparrow_1^{j-1})&\sum_{i=0}^{\lfloor j/3\rfloor} X^{j-3i}Y^{3i}+Z^j \,cs&yz,xz,x^3-y^3,x^{j-2}y,x^j-z^j\\
\hline
9&((j+1)_0,3\uparrow_1^{j-3},2_{1,{j-2}})& X^j+X^2Y^{j-2}+XY^{j-2}Z+Y^{j-2}Z^2&xz-z^2,xy-yz,z^3,y^{j-1},x^j-y^{j-2}z^2\\
\hline
10&((j+1)_0,3\uparrow_1^{j-3},1_{1,2,j-2,j-1})& X^j+X^2Y^{j-2}+XY^{j-1}+Y^{j-1}Z &xz,z^2,x^2y-y^2z,x^{j-1}-xy^{j-2}+y^{j-2}z, y^j\\
\hline
11&((j+1)_0,2\uparrow_1^{j-2},1\uparrow_1^{j-1})& X^j+XY^{j-1}+Z^j\, cs&yz,xz,x^2y,x^{j-1}-y^{j-1},xy^{j-1}-z^j\\
\hline
12&((j+1)_0,(1\uparrow_1^{j-1})^2,1\uparrow_2^{j-2})& X^j+Z^j+ZY^{j-1}\, cs&xz,yz,yz^2,y^{j-1}-z^{j-1},x^j-z^j\\
\hline
\hline
13&(j\uparrow_0^1,(j-1)_1,(j-3)_2)& X^{j-1}Y+X^{j-2}(Y^2+Z^2)\,\,CI&yz,y^2-z^2,x^{j-1}-x^{j-2}y+x^{j-3}z^2\\
\hline
14&(j\uparrow_0^1,(j-1)_1,1\uparrow_2^{j-2})& X^{j-1}Y+X^{j-2}Z^2 +Z^j&yz,y^2,x^2y-xz^2,x^{j-2}z-z^{j-1},x^j \\
\hline
15&(j\uparrow_0^1,(j-2)_{1,2})& X^{j-1}Y+X^{j-3}Z^3 &yz,y^2,x^2y-z^3,z^4,x^{j-2}z,x^j\\
\hline
16&(j\uparrow_0^1,(j-3)_2,1\uparrow_1^{j-1})& \sum_{i=0}^{\lfloor (j-1)/3\rfloor} X^{j-1-3i}Y^{3i+1}+Z^j\, cs&yz,xz,x^3-y^3,x^{j-3}y^2,x^{j-1}y-z^j\\
\hline
17&(j\uparrow_0^1,2\uparrow_1^{j-2})& X^{j-1}Y+XZ^{j-1}&yz,y^2,x^2z,x^{j-2}y-z^{j-2},x^j\\
\hline
18&(j\uparrow_0^1,2\uparrow_2^{j-3},1_{1,2,j-2,j-1})& X^{j-1}Y+XY^{j-1}+Y^{j-1}Z&xz,z^2,xy^2-y^2z,x^{j-1}-xy^{j-2},y^j\\
\hline
19&(j\uparrow_0^1,1_1,(1\uparrow_2^{j-2})^2,1_{j-1})& X^{j-1}Y+Y^j+Z^j\, cs&yz,xz,xy^2,x^{j-1}-y^{j-1},y^j-z^j\\
\hline
\hline
20&((j-1)_0,((j-1)_1)^2,(j-1)_2)& X^{j-2}(Y^2+Z^2)\,\,CI&yz,y^2-z^2,x^{j-1}\\
\hline
21&((j-1)\uparrow_0^2,1\uparrow_1^{j-1})& X^{j-2}Y^2+Z^j\, cs&yz,xz,y^3,x^{j-1},x^{j-2}y^2-z^j\\
\hline
\hline
22&(4\uparrow_0^{j-3},2_1,2_{j-2})& X^3Y^{j-3}+XY^{j-2}Z\,\,CI&x^2-yz,z^2,y^{j-1}\\
\hline
23&(4\uparrow_0^{j-3},1_{1,2,j-2,j-1})& X^3Y^{j-3}+Y^j+Y^{j-1}Z&xz,z^2,x^3-y^2z,xy^{j-2},y^j-y^{j-1}z\\
\hline
\hline
24&(3\uparrow_0^{j-2},1\uparrow_1^{j-1})& X^2Y^{j-2}+Z^j\, cs&yz,xz,x^3,y^{j-1},x^2y^{j-2}-z^j\\
\hline\hline
25&(2\uparrow_0^{j-1},2\uparrow_1^{j-2})& X(Y^{j-1}+Z^{j-1})\,\,CI&yz,x^2,y^{j-1}-z^{j-1}\\
\hline
26&(2\uparrow_0^{j-1},1_1,(1\uparrow_2^{j-2})^2,1_{j-1})& XY^{j-1}+Y^{j-3}Z^3&xz,x^2,xy^2-z^3,y^{j-2}z,y^j\\
\hline
\end{array}$
        \caption{Complete 26 JDT for $T=(1,3,4^k,3,1), k\ge 3$ with the corresponding Macaulay dual generator (Theorem \ref{s=4thm}).}\label{4ktable}\end{center}
\end{table*}
\end{landscape}
\newpage

\normalsize
\begin{landscape}
\begin{table*}[h]
\begin{center}
$\begin{array}{|c|c|c|c|}
\hline
&\mathcal{S}&F&I=\Ann(F)\\
\hline
\hline
1&(5_0,(3_1)^2,(1_2)^2)& X^4+X^3Z+2X^2Y^2+XZ^3+Y^4&yz, x^3-xz^2-z^3,x^2y-2y^3,x^2z-z^3, xy^2-2z^3 \\
\hline
2&(5_0,3_1,2\uparrow_1^2,1_2)&X^4+X^2Z^2+XY^3&yz,x^2y, x^3-xz^2, y^3-xz^2,z^3  \\
\hline
3&(5_0,3_1,1\uparrow_1^3,(1_2)^2)& X^4+X^2Z^2+YZ^3+Y^4&xy, x^2z-yz^2,x^3-xz^2, y^3-z^3\\
\hline
4&(5_0,(2\uparrow_1^2)^2)& X^4+XYZ^2&y^2, x^3-yz^2,x^2y,x^2z, z^3 \\

\hline
5&(5_0,2\uparrow_1^2,1\uparrow_1^3,1_2)& X^4+XY^3+Y^2Z^2+Z^4&xz, x^2y, x^3-y^3,xy^2-yz^2,y^2z-z^3 \\

\hline\hline
6&(4\uparrow_0^1,3_1,(1_2)^2)&X^3Y+X^2Z^2+Y^3Z&xy-z^2, x^3-y^2z,y^3-x^2z, yz^2,z^3  \\
\hline

7&(4\uparrow_0^1,2_{1,2},1_{2})& X^3Y+XY^2Z+Z^4&x^2-yz, xy^2-z^3,xz^2, yz^2, y^3  \\
\hline

8&(4\uparrow_0^1,1\uparrow_1^3,(1_2)^2)& X^3Y+Y^2Z^2& xz, x^3-yz^2,xy^2,y^3,z^3  \\
\hline
\hline
9&(3\uparrow_0^2,3_{1},1_2)&X^2(Y^2+Z^2)+Y^4&yz,x^3, z^3, xy^2-xz^2, x^2y-y^3\\
\hline
10&(3\uparrow_0^2,2\uparrow_1^2)& X^2Y^2+XZ^3& yz, x^2z, x^3, y^3, xy^2-z^3\\
\hline
11&(3\uparrow_0^2,1\uparrow_1^3,1_{2})& X^2Y^2+YZ^3& xz, y^2z, x^3, y^3, x^2y-z^3\\
\hline
12&(2\uparrow_0^3,2\uparrow_1^2,1_{2})& X(Y^3+Z^3)+Y^2Z^2&x^2, y^2z-xz^2, xyz, y^3-z^3, xy^2-yz^2\\
\hline
\end{array}$
\caption{The 12 JDT for $T=(1,3,5,3,1)$ with the Macaulay dual generator, and ideal  (Theorem \ref{s=5thm}).}\label{13531table}
\end{center}
\end{table*}
\end{landscape}
\newpage
\begin{table*}[h]
\begin{center}
\vskip -0.6cm
$\begin{array}{|c|c|c|c|}
\hline
&\mathcal{S}&F&\text{Table \ref{newk=5table}}\\
\hline\hline
1&(6_0,(4_1)^2,(2_2)^2)&X^5+X^3YZ+XZ^4&1\\
\hline
2&(6_0,(4_1)^2,2_2,1_{2,3})&X^5+X^3YZ+Z^5&4\\
\hline
3&(6_0,(4_1)^2,(1_{2,3})^2)& X^5+X^4Y+X^3Z^2+Y^5+Z^5&3\\
\hline
4&(6_0,4_1,3_{1,2},2_2)&X^5+X^3Y^2+X^2Z^3&2\\
\hline
5&(6_0,4_1,3_{1,2},1_{2,3})&X^5+X^4Y+X^2Z^3+Y^5&5,7\\
\hline
6&(6_0,4_1,2_2,1\uparrow_1^4,1_{2,3})& X^5+X^3Y^2+Y^2Z^3&17\\
\hline
7&(6_0,4_1,2\uparrow_1^3,1_{2,3})& X^5+X^4Y+XZ^4+Y^2Z^3&14\\
\hline
8&(6_0,4_1	,2\uparrow_1^3,2_2)& X^5+X^3Y^2+XZ^4&9\\
\hline
9&(6_0,4_1,1_1,(1\uparrow_2^3)^3,1_4)& X^5+X^4Y+Y^2Z^3&21 \\
\hline

10&(6_0,3_{1,2},2_{1,3},1_{2,3})& X^5+X^2Y^3+XY^3Z&15 \\
\hline

11&(6_0,3_{1,2},2\uparrow_1^3)& X^5+X^2Y^3+XZ^4&10\\
\hline

12&(6_0,(3_{1,2})^2)&X^5+X^2YZ^2&3\\
\hline

13&(6_0,3\uparrow_1^2,1\uparrow_1^4,1\uparrow_2^3)& X^5+X^2Y^3+YZ^4&18\\
\hline
14&(6_0,(2\uparrow_1^3)^2)& X^5+XY^4+XZ^4&16\\
\hline
15&(6_0,2\uparrow_1^3,1\uparrow_1^4,1_{2,3})& X^5+XY^4+YZ^4&22 \\
\hline\hline
16&(5_{0,1},4_1,(2_2)^2)&X^4Y+X^3Z^2+XYZ^3&23\\
\hline
17&(5_{0,1},4_1,2_2,1_{2,3})&X^4Y+X^3Z^2+XZ^4+Y^5&25\\
\hline
18&(5_{0,1},4_1,(1_{2,3})^2)& X^4Y+X^3Z^2+YZ^4&31\\
\hline
19&(5_{0,1},3_{1,2},2_2)&X^4Y+X^2Z^3+XY^4&24\\
\hline
20&(5_{0,1},3\uparrow_1^2,1\uparrow_2^3)& X^4Y+X^2Z^3+YZ^4&28\\
\hline
21&(5_{0,1},2\uparrow_1^3,2_2)& X^4Y+XYZ^3&29\\
\hline
22&(5_{0,1},2_{1,2,3},1_{2,3})& X^4Y+XZ^4+Y^2Z^3&32 \\
\hline
23&(5_{0,1},2_2,1\uparrow_1^4,1_{2,3})& X^4Y+XY^4+Y^2Z^3&33,34\\
\hline
24&(5_{0,1},1_1,(1\uparrow_2^3)^3,1_4)& X^4Y+Y^2Z^3&35\\
\hline\hline

25&(4\uparrow_0^2,3\uparrow_1^2)&X^3Y^2+X^2Z^3&37\\
\hline

26&(4\uparrow_0^2,4_1,2_2)&X^3YZ+XZ^4&36\\
\hline

27&(4\uparrow_0^2,4_1,1\uparrow_2^3)& X^3Y^2+X^3Z^2+Z^5&38\\
\hline
28&(4\uparrow_0^2,2\uparrow_1^3)& X^3Y^2+XZ^4&39\\
\hline
29&(4\uparrow_0^2,2_{1,3},1_{2,3})& X^3Y^2+XY^3Z+Z^5&43 \\
\hline
30&(4\uparrow_0^2,1\uparrow_1^4,1\uparrow_2^3)& X^3Y^2+YZ^4&40,44\\
\hline
\hline
31&(3\uparrow_0^3,3_{1,2})& X^2Y^3+X^2Z^3&\text{none}\\
\hline
32&(3\uparrow_0^3,2\uparrow_1^3)& X^2Y^3+XZ^4&45\\
\hline
33&(3\uparrow_0^3,1\uparrow_1^4,1_{2,3})& X^2Y^3+YZ^4&46\\
\hline\hline

34&(2\uparrow_0^4,2\uparrow_1^3,2_2)& XY^2Z^2&\text{none}\\
\hline
35&(2\uparrow_0^4,2\uparrow_1^3,1_{2,3})& XY^4+XZ^4+Y^2Z^3&47\\
\hline
\end{array}$
\caption{The 35 JDT for $T=(1,3,5,5,3,1)$, with Macaulay dual generator and reference to
corresponding entry of Table \ref{newk=5table} (Theorem~\ref{s=5thm}).}\label{55table}
\end{center}
\end{table*}
\newpage
\small
\begin{landscape}
\begin{table*}
\tiny
\begin{center}
$\begin{array}{|c|c|c|c|c|c|c|c|}\hline
\text {\#}& \mathcal{S}&F&I=\Ann(F)\\\hline\hline
 1&((j+1)_0,((j-1)_1)^2,((j-3)_2)^2)&X^j+X^{j-2}YZ+X^{j-4}Z^4&y^2, yz^2, x^2y-z^3, x^{j-3}z^2, x^{j-1}-x^{j-5}z^4\\
\hline
2&((j+1)_0,(j-1)_1,(j-2)_{1,2},(j-3)_2)&X^j+X^{j-2}Y^2+X^{j-3}Z^3&yz, y^3, xy^2-z^3, x^{j-2}z, x^{j-1}-x^{j-4}z^3\\
\hline
 3&((j+1)_0,((j-2)_{1,2})^2)&\ast &xy-yz, x^3-z^3,y^3-x^2z-y^2z-xz^2+yz^2-z^3, x^{j-7}yz^5 ,x^{j-5}z^4\\
\hline
4&((j+1)_0,((j-1)_1)^2,(j-3)_{2},1\uparrow_2^{j-2})&X^j+X^{j-2}YZ+Y^j&z^2, y^2z, xy^2, y^{j-1}-x^{j-2}z, x^{j-1}-x^{j-3}yz\\
\hline
 5&((j+1)_0,(j-1)_1,(j-2)_{1,2},1\uparrow_2^{j-2})&X^j+X^{j-1}Y+X^{j-3}Z^3+Y^j&yz,xy^2,x^2y-z^3,x^{j-2}z,x^{j-1}-y^{j-1}-x^{j-4}y^3\\
\hline
 6&((j+1)_0,((j-1)_1)^2,2\uparrow_2^{j-3},1_{2,j-3})&X^j+X^{j-1}Y+X^{j-2}Z^2+XZ^{j-1}+YZ^{j-1}&y^2,xyz,x^2y-xz^2+yz^2,x^{j-2}x-yz^{j-2},x^{j-1}-z^{j-1}\\
\hline 
 7&((j+1)_0,(j-1)_1,3\uparrow_1^{j-3},1_{2,j-2})&X^j+X^{j-1}Y+X^2Z^{j-2}+YZ^{j-1}&y^2,xyz,x^2z-yz^2,x^{j-2}y-xz^{j-2},x^{j-1}-xz^{j-2}-z^{j-1}\\
\hline
 8&((j+1)_0,4\uparrow_1^{j-4},3_{1,j-3},1_{2,j-2}) &X^j+Y^{j-3}(X^3+X^2Z+XZ^2+Y^2Z+Z^3)&xz-z^2,xy^2-y^2z+z^3,x^2y-yz^2,y^{j-2}z-y^{j-4}z^3,x^{j-1}+y^{j-1}-y^{j-3}z^2\\
\hline
 9&((j+1)_0,(j-1)_1,(j-3)_2,2\uparrow_1^{j-1})& X^j+X^{j-2}Y^2+XZ^{j-1}&yz, x^2z,y^3,x^{j-3}y^2-z^{j-1},x^{j-1}-z^{j-1}\\
\hline
 10& ((j+1)_0,(j-2)_{1,2},2\uparrow_1^{j-2})& X^j+X^{j-1}Y+X^{j-2}Y^2+XZ^{j-1}&yz,x^2z,y^3,x^{j-3}y^2-z^{j-1},x^{j-1}-x^{j-2}y\\
\hline
 11&((j+1)_0,(j-1)_1,3\uparrow_2^{j-4},2_{1,2,j-3,j-2})& X^j+X^{j-1}Y+X^2Y^{j-2}+XY^{j-2}Z+Y^{j-2}Z^2&xz-z^2, z^3, xy^2-y^2z, y^{j-1}, x^{j-1}-x^{j-2}y-y^{j-3}z^2 \\
\hline
 12&((j+1)_0,4\uparrow_1^{j-4},2_{1,2,j-3,j-2})& X^j+X^3Y^{j-3}+XY^{j-2}Z&z^2,x^2y-y^2z,x^2z,y^{j-1}, x^{j-1}-y^{j-2}z\\
\hline 
 13&(j+1)_0,((j-1)_1)^2,(1\uparrow_2^{j-2})^2)& X^j+X^{j-1}Y+X^{j-2}Z^2+Y^j+Z^j&yz,yx^2-z^2x,xy^2,x^{j-2}z-z^{j-1},x^{j-1}-y^{j-1}-x^{j-1}z^2\\
\hline
 14&((j+1)_0,(j-1)_1,2\uparrow_1^{j-2},1\uparrow_2^{j-2})&X^j+X^{j-1}Y+XZ^{j-1}+Y^j&yz,x^2z,xy^2,x^{j-2}y-z^{j-1},x^{j-1}-y^{j-1}-z^{j-1}\\
\hline
 15&(((j+1)_0,3\uparrow_1^{j-3},2_{1,j-2},1\uparrow_2^{j-2})& X^j+X^2Y^{j-2}+XZY^{j-2}&z^2,x^2z,x^2y-xyz,y^{j-1}, x^{j-1}-zy^{j-2}\\
\hline
 16&((j+1)_0,(2\uparrow_1^{j-2})^2)& X^j+XY^{j-1}+XZ^{j-1}&yz,x^2y,x^2z,x^{j-1}-y^{j-1},x^{j-1}-z^{j-1}\\
\hline
 17&((j+1)_0,(j-1)_1,(j-3)_2,1\uparrow_1^{j-1},1\uparrow_2^{j-2})&X^j+X^{j-2}Y^2+YZ^{j-1}&xz,y^2z,y^3,x^{j-2}y-z^{j-1},x^{j-1}-x^{j-3}y^2\\
\hline
 18& ((j+1)_0,(j-2)_{0,1},1\uparrow_1^{j-2},1\uparrow_2^{j-3})& X^j+X^{j-1}Y+X^{j-2}Y^2+YZ^{j-1}&xz,y^2z,y^3,x^{j-2}y-x^{j-3}y^2-z^{j-1},x^{j-1}-x^{j-2}y^2-z^{j-1}\\
\hline
 19&((j+1)_0,(j-1)_1,2\uparrow_2^{j-3},1\uparrow_1^{j-1},1_2,1_{j-2})&X^j+X^{j-1}Y+XY^{j-1}+Y^{j-2}Z^2&xz,xy^2-yz^2,z^3, x^{j-2}y-y^{j-1},x^{j-1}-y^{j-1}-y^{j-3}z^2\\
\hline 
 20&((j+1)_0,3\uparrow_1^{j-3},1\uparrow_1^{j-1},1_2,1_{j-2})&X^j+X^{2}Y^{j-2}+Y^{j-1}Z+Z^j&xz,yz^2,x^2y-y^2z,y^{j-1}-z^{j-1}, x^{j-1}-xy^{j-2}\\
\hline 
 21&((j+1)_0,(j-1)_1,1_{1,j-1},(1\uparrow_2^{j-2})^3)&X^j+X^{j-1}Y+Y^{j-2}Z^2&xz,z^3,xy^2,y^{j-1},x^{j-1}-x^{j-2}y-y^{j-3}z^2\\
\hline
 22&((j+1)_0,2\uparrow_1^{j-2},1_{1,j-1},(1\uparrow_2^{j-2})^2)&X^j+XY^{j-1}+YZ^{j-1}&xz,z^3,xy^2,y^{j-1},x^{j-1}-x^{j-2}y-y^{j-3}z^2\\
\hline
23&(j_{0,1},(j-1)_1,((j-3)_2)^2)&X^{j-1}(Y+Z)+X^{j-2}YZ+X^{j-4}Y^4 &z^2, y^2z, y^3-x^2z+xyz, x^{j-3}y^2, x^{j-1}-x^{j-2}y-x^{j-2}z+2x^{j-3}yz\\
\hline
 24&(j_{0,1},(j-2)_{1,2},(j-3)_2)&X^{j-1}Y+X^{j-2}Y^2+X^{j-3}Z^3&yz,y^3,xy^2-z^3,x^{j-2}z, x^{j-1}-x^{j-2}y+x^{j-3}z^2\\
\hline
 25&(j_{0,1},(j-1)_1,(j-3)_2,1\uparrow_2^{j-2})&X^{j-1}Y+X^{j-2}(Y^2+Z^2)+Z^j&yz,y^3,xy^2-xz^2,x^{j-2}z-z^{j-1}, x^{j-1}-x^{j-2}y+x^{j-3}z^2\\
\hline
 26&(j_{0,1},(j-2)_{1,2},1\uparrow_2^{j-2})&X^{j-1}Y+X^{j-3}Z^3+Y^j&yz,xy^2,x^2y-z^3,x^{j-2}z, x^{j-1}-y^{j-1}\\
\hline
 27&(j_{0,1},(j-1)_1,2\uparrow_2^{j-3},1_{2,j-2})&X^{j-1}Y+X^{j-2}Z^2+XY^{j-1}+Y^j&yz,z^3,x^2y-xz^2,xy^{j-2}-y^{j-1}+x^{j-3}z^2, x^{j-1}-y^{j-1}+x^{j-3}z^2\\
\hline
 28&(j_{0,1},3\uparrow_1^{j-3},1_{2,j-2})&X^{j-1}Y+X^2Z^{j-2}+YZ^{j-1}&y^2,xyz,x^2z-yz^2,x^{j-2}y-xz^{j-2}, x^{j-1}-z^{j-1}\\ 
\hline
 29&(j_{0,1},(j-3)_2,2\uparrow_1^{j-2})&X^{j-1}Y+X^{j-2}Y^2+XZ^{j-1}&yz, x^2z, y^3, x^{j-3}y^2-z^{j-1}, x^{j-1}-x^{j-2}y+z^{j-1}\\
\hline
 30&(j_{0,1},3\uparrow_2^{j-4},2_{1,2,j-3,j-2})&X^{j-1}Y+Y^{j-2}(X^2+XZ+Z^2)&xz-z^2,z^3,xy^2-y^2z,y^{j-1},x^{j-1}-y^{j-3}z^2\\
\hline
 31&(j_{0,1},(j-1)_1),(1\uparrow_2^{j-2})^2)&X^{j-1}Y+X^{j-2}Z^2+Y^j+Z^j&yz,xy^2,x^2y-xz^2,x^{j-2}z-z^{j-1},x^{j-1}-y^{j-1}\\
\hline
32&(j_{0,1},2\uparrow_1^{j-2},1\uparrow_2^{j-2})&X^{j-1}Y+XZ^{j-1}+Y^j&yz,x^2z,xy^2,x^{j-2}y-z^{j-1},x^{j-1}-y^{j-1}\\
\hline
 33&(j_{0,1},(j-3)_2,1\uparrow_1^{j-1},1\uparrow_2^{j-2})&X^{j-1}Y+X^{j-2}Y^2+YZ^{j-1}&xz,y^2z,y^3,x^{j-2}y-x^{j-3}y^2-z^{j-1},x^{j-1}-z^{j-1}\\
 \hline
 34&(j_{0,1},2\uparrow_2^{j-3},1\uparrow_1^{j-1},1_{2,j-2}) &X^{j-1}Y+XY^{j-1}+Y^{j-2}Z^2&xz,z^3,xy^2-yz^2,x^{j-2}y-y^{j-1},x^{j-1}-y^{j-1}z^2\\
\hline
 35&(j_{0,1},(1_2^{j-2})^3,1_{1,j-1})&X^{j-1}Y+Y^j+YZ^{j-1}&xz,y^2z,xy^2,y^{j-1}-z^{j-1},x^{j-1}-z^{j-1}\\
\hline
 36&((j-1)_0,((j-1)_1)^2,(j-1)_2,(j-3)_2)&X^{j-2}YZ+X^{j-4}Z^4&y^2,yz^2,x^2y-z^3,x^{j-3}z^2,x^{j-1}\\
\hline
 37&((j-1)_{0,1,2},(j-2)_{1,2})&X^{j-2}Y^2+X^{j-3}Z^3&yz,y^3,xy^2-z^3,x^{j-2}z,x^{j-1}\\
\hline
 38&((j-1)_0,((j-1)_1)^2,(j-1)_2,1\uparrow_2^{j-2})&X^{j-2}YZ+Z^j&y^2,yz^2,xz^2,x^{j-2}y-z^{j-1},x^{j-1}\\
\hline
 39&((j-1)_{0,1,2}, 2\uparrow_1^{j-2})& X^{j-2}Y^2+XZ^{j-1}&yz,y^3,x^2z,x^{j-1},x^{j-3}y^2-z^{j-1}\\
 \hline
40&((j-1)_{0,1,2}, 1\uparrow_1^{j-1},1\uparrow_2^{j-2} & X^{j-2}Y^2+YZ^{j-1}&xz,y^3,y^2z,x^{j-1},z^{j-2}y-z^{j-1} \\\hline
\end{array}$
\caption{The JDT \#1-40 for $T=(1,3,5^k,3,1), j=k+3, |T|=5j-7, k\ge 4$ (Theorem~\ref{s=5thm}).} \label{newk=5table}
\end{center}
\end{table*}
\end{landscape}
\newpage
\begin{landscape}
\vskip -1.5cm
\begin{table*}
\small
\begin{center}
$\begin{array}{|c|c|c|c|c|c|c|c|}\hline
\text {\#}& \mathcal{S}&F&I=\Ann(F)\\\hline\hline
 41&(5\uparrow_0^{j-4},3_{1,j-3},1_{2,j-2})&Y^{j-4}(X^4+X^2YZ+Y^2Z^2)&x^2-yz, z^3, xz^2, y^{j-1},xy^{j-2}\\
\hline 
 42&(5\uparrow_0^{j-4},2_{1,2,j-3,j-2})&X^4Y^{j-4}+XY^{j-2}Z&z^2,x^2z,x^3-y^2z,y^{j-1},x^2y^{j-3}\\
\hline
 43&(4\uparrow_0^{j-3},2_{1,j-2},1\uparrow_2^{j-2})&X^3Y^{j-3}+XY^{j-2}Z+Z^j&x^2-yz,xz^2,yz^2,y^{j-1},xy^{j-2}-z^{j-1}\\
 \hline
 44&(4\uparrow_0^{j-3},1\uparrow_1^{j-1},1_{2,j-2})&X^3Y^{j-3}+(Y+Z)^{[j]}-Y^j-ZY^{j-1}&xz,yz^2-z^3,x^3+y^2z-z^3,y^{j-1},xy^{j-2}\\
\hline
45&(3\uparrow_0^{j-2},2\uparrow_1^{j-2}) &X^2Y^{j-2}+XZ^{j-2}&yz,x^2z, x^3, y^{j-1},xy^{j-2}-z^{j-1} \\
\hline 
46&(3\uparrow_0^{j-2},1\uparrow_1^{j-1},1\uparrow_2^{j-2})&X^2Y^{j-2}+YZ^{j-1}&xz,y^2z,x^3,y^{j-1},x^2y^{j-3}-z^{j-1}\\
\hline 
47&(2\uparrow_0^{j-1},2\uparrow_1^{j-2},1\uparrow_2^{j-2})&XYZ^{j-2}+Y^j&x^2,y^2z,xy^2,z^{j-1},y^{j-1}-xz^{j-2}\\
\hline 
\end{array}$
\caption{The JDT \#41-47 for $T=(1,3,5^k,3,1), j=k+3, |T|=5j-7, k\ge 4$.} \label{newk=5btable}
\end{center}
\end{table*}
\vskip -1.5cm
\end{landscape}
\listoftables

\end{document}